\documentclass[sn-mathphys]{sn-jnl}

\jyear{2021}%

\usepackage{amssymb,amsfonts,amsmath}
\usepackage[mathscr]{eucal}
\usepackage{multirow}
\usepackage{makecell}
\usepackage{amsbsy}
\usepackage{graphicx}
\usepackage{stmaryrd}
\usepackage{bm}
\usepackage{braket}  
\usepackage{xparse}
\usepackage{xspace}
\usepackage{xcolor}
\usepackage{algorithm, algpseudocode}
\usepackage{footnote}
\makesavenoteenv{algorithm} 
\usepackage{booktabs}
\usepackage[caption=false]{subfig}
\usepackage{hyperref}
\usepackage{cleveref} 

\NewDocumentCommand \tensor {O{}m} {\boldsymbol{#1\mathscr{\MakeUppercase{#2}}}} 
\newcommand{\mat}[1]{\mathbf{#1}}
\newcommand{\vect}[1]{\bm{#1}}

\newcommand{\mc}[1]{\mathcal{#1}}

\newcommand{\bb}[1]{\mathbb{#1}}
\catcode`\|=12
\DeclareMathOperator{\diag}{diag}

\theoremstyle{thmstyleone}%
\newtheorem{theorem}{Theorem}%
\newtheorem{lemma}[theorem]{Lemma}

\theoremstyle{thmstyletwo}%
\newtheorem{remark}{Remark}%

\theoremstyle{thmstylethree}%
\newtheorem{definition}{Definition}%

\begin{document}
\begin{sloppypar}
\title[A Block-Randomized Stochastic Method for CP Decomposition]{A Block-Randomized Stochastic Method with Importance Sampling for CP Tensor Decomposition }


\author[1]{\fnm{Yajie} \sur{Yu}}\email{zqyu@cqu.edu.cn}

\author*[1]{\fnm{Hanyu} \sur{Li}}\email{lihy.hy@gmail.com or hyli@cqu.edu.cn}


\affil*[1]{\orgdiv{College of Mathematics and Statistics}, \orgname{Chongqing University}, \city{Chongqing}, \postcode{401331}, \country{P.R. China}}



\abstract{One popular way to compute the CANDECOMP/PARAFAC (CP) decomposition of a tensor is to transform the problem into a sequence of overdetermined least squares subproblems with Khatri-Rao product (KRP) structure involving factor matrices. In this work, based on choosing the factor matrix randomly, we propose a mini-batch stochastic gradient descent method with importance sampling for those special least squares subproblems. Two different sampling strategies are provided. They can avoid forming the full KRP explicitly and computing the corresponding probabilities directly. 
The adaptive step size version of the method is also given. For the proposed method, we present its detailed theoretical properties and comprehensive numerical performance. The results on synthetic and real data show that our method performs better than the corresponding  one in the literature. 
}



\maketitle

\section{Introduction}
\label{sec:introduction}
The \emph{CP decomposition} of a tensor factorizes the tensor into a sum of $R$ rank-one tensors. That is, given an $N$-way tensor $\tensor{X}$ of size $I_1 \times I_2 \times \cdots \times I_{N}$, we wish to write it as:
\begin{equation*}
\label{eq:cp}
\tensor{X} = \llbracket \mat{A}^{(1)},\mat{A}^{(2)}, \cdots,\mat{A}^{(N)} \rrbracket=\sum_{r=1}^{R} \vect{a}^{(1)}_r \circ \vect{a}^{(2)}_r  \circ \cdots \circ \vect{a}^{(N)}_r,
\end{equation*}
where $R$ is a positive integer, $\vect{a}^{(n)}_r \in \bb{R}^{I_n}$, 
and $\mat{A}^{(n)}=[\vect{a}^{(n)}_1, \vect{a}^{(n)}_2, \cdots, \ \vect{a}^{(n)}_R] \in \bb{R}^{I_n \times R}$ for $n=1, \cdots,N$. 
Usually, $\mat{A}^{(n)}$ is called the $n$-th \emph{factor matrix}. The CP decomposition is an important tool for data analysis and has found many important applications in some fields such as chemometrics, biogeochemistry, neuroscience, signal processing, and cyber traffic analysis; see e.g., \cite{kolda2009TensorDecompositions,sidiropoulos2017TensorDecomposition}.

It is well known that the computation of CP decomposition is a challenging problem. Currently, there are many methods for this decomposition. A popular one is the alternating least squares (ALS) method proposed in the original papers \cite{carroll1970AnalysisIndividual,harshman1970FoundationsPARAFAC}. 
Specifically, we transform the $n$-th factor matrix $\mat{A}^{(n)}$ as the solution to the following least squares problem:
\begin{equation}
\label{eq:lsq_krp}
\begin{gathered}
\min_{\mat{A} \in \bb{R}^{I_n \times R}}  \| \mat{Z}^{(n)} \mat{A}^\intercal - \mat{X}_{(n)}^\intercal \|_F^2,
\end{gathered}
\end{equation}
where $\mat{Z}^{(n)} = \mat{A}^{(N)} \odot \cdots \odot \mat{A}^{(n+1)} \odot \mat{A}^{(n-1)} \odot \cdots \odot \mat{A}^{(1)} \in \bb{R}^{J_n \times R}$ with the symbol $\odot$ denoting the KRP, $\mat{X}_{(n)} \in \bb{R}^{I_n \times J_n}$ is the mode-$n$ unfolding of the input tensor $\tensor{X}$, and $J_n=\prod_{m=1, m \ne n}^{N} I_m$.
Here, the \emph{mode-$n$ unfolding} of a tensor means aligning the mode-$n$ fibers as the columns of an $I_n \times J_n$ matrix and the relation between the index of the tensor entry $x_{i_1,i_2, \cdots,i_N}$ and the index of the matrix entry $x_{i_n, j}$ is
\begin{equation}
\label{eq:map}
j=1+\sum_{k=1, k\ne n}^{N} (i_k-1) J'_k,  
\end{equation}
where $J'_k=\prod_{m=1, m \ne n}^{k-1} I_m$.
More symbol definitions are consistent with \cite{kolda2009TensorDecompositions}.

As we know, the ALS method is the  ``workhorse" method for CP decomposition (CP-ALS). However, for large-scale problems, the cost of the method is prohibitive.
To reduce the cost, Battaglino et al. \cite{battaglino2018PracticalRandomized} applied random projection and uniform sampling techniques of the regular least squares problem \cite{drineas2011FasterLeast} to \eqref{eq:lsq_krp} and designed the corresponding randomized algorithms in 2018. A main and attractive feature of the algorithms in \cite{battaglino2018PracticalRandomized} is that they never explicitly form the full KPR matrices when applying projection and sampling.
Later, building on the uniform sampling technique in \cite{battaglino2018PracticalRandomized}, Fu et al. \cite{fu2020BlockRandomizedStochastic} utilized the mini-batch stochastic gradient descent (SGD) method to solve the least squares subproblems in CP-ALS. This method was recently extended to the momentum version \cite{wang2021momentum}.
In 2020, Larsen and Kolda \cite{larsen2020PracticalLeverageBased} performed the leverage-based sampling for \eqref{eq:lsq_krp} without forming the full KRP matrices and computing the corresponding probabilities directly. By the way, the method for estimating the leverage scores of the KRP matrix without forming it explicitly also appears in \cite{cheng2016SPALSFast}.
The random sampling methods introduced in \cite{battaglino2018PracticalRandomized,fu2020BlockRandomizedStochastic,larsen2020PracticalLeverageBased} are built on fiber sampling. Besides, there are also some other random sampling algorithms for CP decomposition built on element sampling \cite{vervliet2014BreakingCurse,bhojanapalli2015NewSampling,vu2015NewStochastic} or sub-tensor sampling \cite{beutel2014FlexiFaCTScalable,vervliet2016RandomizedBlock}. However, these two samplings are not suitable for the case of constraint \cite{fu2020BlockRandomizedStochastic}. 

Inspired by the doubly randomized computational framework in \cite{fu2020BlockRandomizedStochastic} and the leverage-based sampling method in \cite{larsen2020PracticalLeverageBased}, in this work, we present a mini-batch SGD method with importance sampling for CP decomposition.
On the basis of the leverage scores or the squared Euclidean norms about rows of the KRP matrices, we propose two sampling strategies to select the mini-batch for the SGD method. 
As in \cite{larsen2020PracticalLeverageBased}, these two strategies don't need to form the KRP matrices explicitly and compute the corresponding probabilities directly
 either.
Since the rows sampled by importance sampling contain more information compared to those by uniform sampling \cite{needell2016StochasticGradient}, 
our new method can converge faster than the one from \cite{fu2020BlockRandomizedStochastic}. 
Extensive numerical experiments validate this result.

The remainder of this paper is organized as follows.  
\Cref{sec:background} provides some preliminaries. 
In \Cref{sec:samp_alg}, we present the sampling strategies and our method and its adaptive step size
version. 
The relevant theoretical properties are given in \Cref{sec:Convergence_Properties}.
\Cref{sec:Numerical_Results} is devoted to numerical experiments to illustrate our method. Finally, the concluding remarks of the whole paper are presented.

\section{Preliminaries}
\label{sec:background}
We first introduce the idea on sampling the rows without forming the full  KRP given in \cite{battaglino2018PracticalRandomized}, that is, how to compute the sampled matrix $\mat{Z}^{(n)}(\mc{F}_n, :)$ without 
explicitly forming $\mat{Z}^{(n)}$, where the set $\mc{F}_n \subset \{1, \cdots, J_n\}$  contains the 
indices of the rows sampled from $\mat{Z}^{(n)}$ and we will use the shorthand notation $\mat{Z}^{(n)}_{\mc{F}_n}=\mat{Z}^{(n)}(\mc{F}_n, :)$ later in this paper. 
The idea mainly comes from the structure of $\mat{Z}^{(n)}$ and the definition of KRP,  
which make 
the 
rows of $\mat{Z}^{(n)}$ 
be written as the Hadamard product of the corresponding rows of the factor matrices, i.e.,
\begin{align}
\label{eq:index_map}
\mat{Z}^{(n)}(j,:) &= \mat{A}^{(1)}(i_1,:) \circledast \cdots \circledast \mat{A}^{(n-1)}(i_{n-1},:)\nonumber \\
&\quad \circledast \mat{A}^{(n+1)}(i_{n+1},:) \circledast \cdots \circledast \mat{A}^{(N)}(i_{N},:),
\end{align}
and the index $j$ and the \emph{multi-index} $(i_1,\cdots,i_{n-1},i_{n+1},\cdots,i_N)$ be related as in \eqref{eq:map}.
Based on the above fact,  Battaglino et al. \cite{battaglino2018PracticalRandomized} presented  \Cref{alg:skr} for computing $\mat{Z}^{(n)}_{\mc{F}_n}$, i.e., the \emph{sampled KRP} (SKRP). The notation $\vect{idx} \in \bb{R}^{|\mc{F}_n| \times (N-1)}$ in \Cref{alg:skr} represents the set of $|\mc{F}_n|$ multi-indices, that is, a row in $\vect{idx}$ represents a multi-index:
\begin{equation}
\label{eq:idx}
\{i_1^{(j)},\cdots,i_{n-1}^{(j)},i_{n+1}^{(j)},\cdots,i_N^{(j)}\} 
~~ \text{with} ~~ j \in \mc{F}_n.
\end{equation}
Here, we assume that these tuples are stacked in matrix
form for efficiency. Thus, each multiplicand $\mat{A}_{|\mc{F}_n|}^{(n)}$ is of size
$|\mc{F}_n| \times R$. Furthermore, \Cref{alg:skr} also presents how to compute $\mat{X}_{(n)}^{\mc{F}_n} = \mat{X}_{(n)}^\intercal (\mc{F}_n, :) = \mat{X}_{(n)}(:, \mc{F}_n)$ given $\vect{idx}$. Note that, to find $\mat{Z}^{(n)}_{\mc{F}_n}$ and $\mat{X}_{(n)}^{\mc{F}_n}$, we don't need to form ${\mc{F}_n}$ explicitly either. 

\begin{algorithm}
	\caption{SKRP and Sampled tensor fibers (SKRP-ST) \cite{battaglino2018PracticalRandomized}}
	\label{alg:skr}
	\begin{algorithmic}[1]\footnotesize
		\Function{$[\mat{Z}^{(n)}_{\mc{F}_n}, \mat{X}_{(n)}^{\mc{F}_n}]$ = SKRP-ST}{$n$, $\vect{idx}$, $\{\mat{A}^{(k)}\}_{k=1, k \ne n}^{N}$}
		\State $\mat{Z}^{(n)}_{\mc{F}_n} \gets \textbf{1}$ 
		\Comment $\textbf{1} \in \bb{R}^{|\mc{F}_n| \times R}$ a matrix with all elements 1
		\For{$m=1,\cdots,n-1,n+1,\cdots,N$}
		\State $\mat{A}_{|\mc{F}_n|}^{(m)} \gets \mat{A}^{(m)}(\vect{idx}(:, m),:)$  \label{algstep:idx} 
		\State $\mat{Z}^{(n)}_{\mc{F}_n} \gets \mat{Z}^{(n)}_{\mc{F}_n} \circledast \mat{A}_{|\mc{F}_n|}^{(m)}$
		\EndFor
		\State $\tensor{X}^{\mc{F}_n} \gets
		\tensor{X}(\vect{idx}(:,1),\cdots,\tensor{X}(\vect{idx}(:,n-1),:,\tensor{X}(\vect{idx}(:,n+1),\cdots,\tensor{X}(\vect{idx}(:,N))$
		\State $\mat{X}_{(n)}^{\mc{F}_n} \gets$ \textsc{Unfolding}($\tensor{X}^{\mc{F}_n}$, $n$)
		\State \Return $\mat{Z}^{(n)}_{\mc{F}_n},\mat{X}_{(n)}^{\mc{F}_n}$
		\EndFunction
	\end{algorithmic}
\end{algorithm}


Combining  \Cref{alg:skr} with uniform sampling, i.e., the indices in each column of $\vect{idx}$ are sampled from the corresponding index set uniformly, and the SGD method, Fu et al. \cite{fu2020BlockRandomizedStochastic} proposed the {\it \textbf{B}lock-\textbf{Ra}ndomized \textbf{S}GD for \textbf{CP} \textbf{D}ecomposition} (\texttt{BrasCPD}), where the block-randomized means choosing the factor matrix, i.e., sampling a mode from
all modes, randomly. Specifically, the authors first rewrite the CP decomposition of a tensor as the following optimization problem:
\begin{equation}
\label{eq:certain_criterion}
\min_{\{\mat{A}^{(n)}\in \bb{R}^{I_n \times R}\}_{n=1}^N }~f(\mat{A}^{(1)},\cdots,\mat{A}^{(N)}),
\end{equation}
where
\begin{align*}
\label{eq:certain_criterion_ls}
f(\mat{A}^{(1)},\cdots,\mat{A}^{(N)}) &= \frac{1}{2} \| \tensor{X} - \llbracket \mat{A}^{(1)},\mat{A}^{(2)}, \cdots,\mat{A}^{(N)} \rrbracket \|_F^2 \\
&= \frac{1}{2} \| \mat{Z}^{(n)} (\mat{A}^{(n)})^\intercal - \mat{X}_{(n)}^\intercal \|_F^2.
\end{align*}
Then, by choosing a mode randomly and obtaining  $\mat{Z}^{(n)}_{\mc{F}_n}$ and $\mat{X}_{(n)}^{\mc{F}_n}$ using \Cref{alg:skr} with uniform sampling, with the GD method, 
they update the latent factor
of the sampled least squares problem of \eqref{eq:lsq_krp}, i.e.,
\begin{equation}
\label{eq:solving_sam_laq}
\mat{A}^{(n)} \gets \arg\min_{\mat{A} \in \bb{R}^{I_n \times R}} \| \mat{Z}^{(n)}_{\mc{F}_n}  \mat{A}^\intercal - \mat{X}_{(n)}^{\mc{F}_n} \|_F^2,
\end{equation}
by
\begin{align*} 
\mat{A}^{(n)}_{(t+1)} \leftarrow \mat{A}^{(n)}_{(t)} - \frac{\alpha_{t}}{|\mc{F}_n|} \left( \mat{A}^{(n)}_{(t)} (\mat{Z}^{(n)}_{\mc{F}_n} )^\intercal \mat{Z}^{(n)}_{\mc{F}_n} - \mat{X}_{(n)}^{\mc{F}_n} \mat{Z}^{(n)}_{\mc{F}_n} \right).
\end{align*}
The specific algorithm is listed in \Cref{alg:fu}. 

\begin{savenotes}
\begin{algorithm}
	\caption{BrasCPD \cite{fu2020BlockRandomizedStochastic}}
	\label{alg:fu}
	\begin{algorithmic}[1]\footnotesize
		\Function{$\{\mat{A}^{(n)}\}_{n=1}^N$= BrasCPD}{$\tensor{X}, R, |\mc{F}_n|, \{ \mat{A}^{(n)}_{(0)} \}_{n=1}^{N}, \{\alpha^{t}\}$} 
		
		\Comment{$N$-way tensor $\tensor{X} \in \bb{R}^{I_1 \times \cdots \times I_N}$;}
		
		\Comment{rank $R$, sample size $|\mc{F}_n|$;}
		
		\Comment{initialization $\{ \mat{A}^{(n)}_{(0)} \}_{n=1}^{N}$, step size $\{\alpha^{t}\}_{t=0,1, \cdots}$}
		\State $t\leftarrow 0$
		\Repeat
		\State Uniformly sample $n$ from $\{1,\cdots,N\}$
		\State \label{line:fu_samp_fn} Sample $\mc{F}_n$ uniformly from $\{1,\cdots, J_n \}$\footnote{The codes from \cite{fu2020BlockRandomizedStochastic} show that the specific implementation is actually the same as what are done in \Cref{alg:skr} with uniform sampling. However, for uniform sampling, we can indeed first determine $\mc{F}_n$ and then find $\vect{idx}$  by \eqref{eq:map}. But, for our importance sampling, this way is unpractical.} 
		\State Compute $\mat{G}^{(n)}_{(t)} = \frac{1}{|\mc{F}_n|} \left( \mat{A}^{(n)}_{(t)} (\mat{Z}^{(n)}_{\mc{F}_n})^\intercal \mat{Z}^{(n)}_{\mc{F}_n} - \mat{X}_{(n)}^{\mc{F}_n} \mat{Z}^{(n)}_{\mc{F}_n} \right) $
		\State Update $\mat{A}^{(n)}_{(t+1)} \leftarrow \mat{A}^{(n)}_{(t)} - \alpha_{t} \mat{G}^{(n)}_{(t)}$, and $\mat{A}^{(n')}_{(t+1)} \leftarrow \mat{A}^{(n')}_{(t)}$ for $n'\neq n$
		\State $t\leftarrow t+1$
		\Until{some stopping criterion is reached}
		\State \Return $\{\mat{A}^{(n)}\}_{n=1}^N$
		\EndFunction
	\end{algorithmic}
\end{algorithm}
\end{savenotes}

In addition, the following two definitions are also necessary throughout the rest of this paper.
\begin{definition}
	We say $\vect{p} \in [0,1]^N$, i.e., an $N$ dimensional vector with the entries being in $[0, 1]$, is a \emph{probability distribution} if 
	$\sum_{i=1}^N p_i = 1$.
\end{definition}

\begin{definition}
	For a random variable $\xi \in [N]$, we say \emph{$\xi \sim \text{\sc multinomial}(\vect{p})$} if $\vect{p} \in [0,1]^N$ is a probability distribution and ${\sf Pr}(\xi = i) = p_i$.
\end{definition}

\section{Sampling Strategies and Proposed Method}
\label{sec:samp_alg}
In this section, we first present two importance sampling strategies to select the mini-batch for
SGD method based on the leverage scores or the squared Euclidean norms about rows of the
KRP matrix, and then present the
doubly randomized variant of CP-ALS. 
Moreover, its adaptive step size version is also presented, which can 
eliminate the challenges in adjusting the parameters and the non-convergence of the algorithm owing to improper step sizes.

\subsection{Sampling strategies}
\label{sec:efficient_sampling}
We begin with some definitions.
\begin{definition}[Leverage Scores \cite{drineas2012FastApproximation}]
	\label{def:leverages_scores}
	Let $\mat{A} \in \bb{R}^{m \times n}$ with $m > n$, and let $\mat{Q} \in \bb{R}^{m \times n}$ be any orthogonal basis for the column space of $\mat{A}$.
	The \emph{leverage score} of the $i$-th row of $\mat{A}$ is given by
	\begin{equation*}
	\ell_i(\mat{A}) = \|\mat{Q}(i,:)\|_2^2.
	\end{equation*}
\end{definition}

\begin{definition}[Leveraged-based Probability Distribution  \cite{woodruff2014SketchingTool}]
	\label{def:leverages_scores_sampling}
	Let $\mat{A} \in \bb{R}^{m \times n}$ with $m > n$.
	We say a probability distribution $\vect{p}=[p_1, \cdots, p_{m} ]^\intercal$ is a \emph{leveraged-based probability distribution} for $\mat{A}$ if $p_i \ge \beta \frac{\ell_i(\mat{A})}{n}$ with $0 < \beta \le 1$ and $i \in [m]$.
\end{definition}

\begin{definition}[Euclidean-based Probability Distribution]
	\label{def:euclidean_norms_sampling}
	Let $\mat{A} \in \bb{R}^{m\times n}$ with $m > n$.
	We say a probability distribution $\vect{p}=[p_1, \cdots, p_{m} ]^\intercal$ is an \emph{Euclidean-based probability distribution} for $\mat{A}$ if $p_i \ge \beta \|\mat{A}(i, :)\|^2_2 / \| \mat{A} \|_F^2$ with $0 < \beta \le 1$ and $i \in [m]$.
\end{definition}

Since it is expensive to compute the leverage scores of a KRP matrix directly, Cheng et al. \cite{cheng2016SPALSFast} presented their upper bounds.

\begin{lemma}[Leverage Score Bound for KRP Matrix \cite{cheng2016SPALSFast}] 
	\label{lem:lev_bound}
	For matrices $\mat{A}^{(k)} \in \bb{R}^{I_k \times R}$ with $I_k > R$ for $k=1, \cdots, K$, let $\ell_{i_k}$ be the leverage score of the $i_k$-th row of $\mat{A}^{(k)}$. Then, for the KRP matrix $\mat{Z} = \mat{A}^{(1)} \odot \mat{A}^{(2)} \odot \cdots \odot \mat{A}^{(K)}$, the leverage score $\ell_{i_1, \cdots, i_K}$ of its $j$-th row corresponding to $\{i_1, \cdots, i_K\}$ satisfies
	\begin{equation*}
	\ell_{i_1, \cdots, i_K} \le \prod_{k=1}^{K} \ell_{i_k}.
	\end{equation*}
\end{lemma}

Hence, 
the leveraged-based sampling probability for the $j$-th row of the above KRP matrix can be set to be
\begin{equation}
\label{eq:krp_lev_prob}
p_j = \frac{1}{R^{K}} \prod_{k=1}^{K} \ell_{i_k}.
\end{equation}
In this case, $\beta$ in \Cref{def:leverages_scores_sampling} is equal to $1/R^{K-1}$.
Furthermore, Larsen and Kolda \cite{larsen2020PracticalLeverageBased} showed that sampling the $j$-th row of $\mat{Z}$ with the above probability can be carried out by sampling the corresponding rows from factor matrices with suitable probabilities independently. The specific result is given in the following lemma. 

\begin{lemma}[\cite{larsen2020PracticalLeverageBased}]
	\label{lem:KRP_sample}
	Let $\mat{A}^{(k)} \in \bb{R}^{I_k \times R}$ for $k=1, \cdots, K$, and $\vect{\ell}(\mat{A}^{(k)})$ be the vector of leverage scores for $\mat{A}^{(k)}$.
	Let
	\begin{equation*}
	i_k \sim \text{\sc multinomial}(\vect{\ell}(\mat{A}^{(k)}) / R)
	~~ \text{for} ~~ k=1, \cdots, K.
	\end{equation*}
	Then, the probability of selecting the multi-index $\{i_1, \cdots, i_K\}$ is $p_j$ in \eqref{eq:krp_lev_prob}.
\end{lemma}

For the case on the squared
Euclidean norms, we have the similar conclusions. That is, the Euclidean-based sampling probability for the $j$-th row of the above KRP matrix can be set to be
\begin{equation}
\label{eq:krp_euc_prob}
p_j = \prod_{k=1}^K \frac{\|\mat{A}^{(k)} (i_{k}, :)\|^2_2}{\| \mat{A}^{(k)} \|_F^2},
\end{equation}
and to sample the $j$-th row of $\mat{Z}$  with the above probability, it suffices to sample the corresponding rows from factor matrices with suitable probabilities independently. These conclusions are guaranteed by the following two lemmas, whose proofs are similar to those of \Cref{lem:lev_bound,lem:KRP_sample}. So we omit them here.

\begin{lemma}
	\label{lem:euc_bound}
	Let $\mat{A}^{(k)} \in \bb{R}^{I_k \times R}$ with $I_k > R$ for $k=1, \cdots, K$. For the KRP matrix $\mat{Z} = \mat{A}^{(1)} \odot \mat{A}^{(2)} \odot \cdots \odot \mat{A}^{(K)}$, the squared Euclidean norms of its $j$-th row corresponding to $\{i_1, \cdots, i_K\}$ satisfies
	\begin{equation*}
	\| \mat{Z} (j, :) \|^2_2 \leq \prod_{k=1}^K \| \mat{A}^{(k)} (i_{k}, :) \|^2_2 .
	\end{equation*}
\end{lemma}

\begin{lemma}
	\label{lem:KRP_sample_euc}
	Let $\mat{A}^{(k)} \in \bb{R}^{I_k \times R}$ for $k=1, \cdots, K$, and  $\vect{p}_{k}$ be the Euclidean-based probability distribution for $\mat{A}^{(k)}$ with $\beta = 1$, i.e., $\vect{p}_{k}(i) = \|\mat{A}^{(k)}(i, :)\|^2_2 / \| \mat{A}^{(k)} \|_F^2$ with $i \in [I_k]$.
	Let
	\begin{equation*}
	i_k \sim \text{\sc multinomial}(\vect{p}_{k})
	~~\text{for}~~ k=1, \cdots, K.
	\end{equation*}
	Then, the probability of selecting the multi-index $\{i_1, \cdots, i_K\}$ is $p_j$ in \eqref{eq:krp_euc_prob}. 
\end{lemma}

In the following, we introduce how to find the mini-batches and compute their sampling probabilities based on the above Leveraged-based and Euclidean-based probability distributions, respectively. 

We first sample $|\mc{F}_n|$ rows from each $\{ \mat{A}^{(k)} \}_{k=1,k \ne n}^{N}$ using the leveraged-based probability distribution for $\mat{A}^{(k)}$ with $\beta = 1$, i.e., $\vect{p}_k=\frac{\vect{\ell}(\mat{A}^{(k)})}{R}$. 
Thus, we can get the $\vect{idx}$:
\begin{equation}
\label{eq:block_map}
\begin{bmatrix}
\{ i_1^{(j_1)} & \cdots & i_{n-1}^{(j_1)} & i_{n+1}^{(j_1)} & \cdots & i_N^{(j_1)} \} \\
\{ i_1^{(j_2)} & \cdots & i_{n-1}^{(j_2)} & i_{n+1}^{(j_2)} & \cdots & i_N^{(j_2)} \} \\
\vdots & \vdots & \vdots & \vdots & \vdots & \vdots \\
\{ i_1^{(j_{|\mc{F}_n|})} & \cdots & i_{n-1}^{(j_{|\mc{F}_n|})} & i_{n+1}^{(j_{|\mc{F}_n|})} & \cdots & i_N^{(j_{|\mc{F}_n|})} \} \\
\end{bmatrix}.
\end{equation}
Then, based on the above set of tuples, using \eqref{eq:map} and \Cref{lem:KRP_sample}, we can obtain the index set $\mc{F}_n = \{ j_1, j_2, \cdots, j_{|\mc{F}_n|} \}$ of the sampled $|\mc{F}_n|$ rows of $\mat{Z}^{(n)}$, and the corresponding sampling probabilities 
\begin{equation}
\label{eq:block_lev_prob}
\vect{p}_{\mc{F}_n}=[p_{j_1}, \cdots, p_{j_{|\mc{F}_n|}}]^\intercal,
\end{equation} 
where $p_{j_f} = \frac{\bar{\ell}_{j_f}(\mat{Z}^{(n)})}{R^{N-1}}$ and $\bar{\ell}_{j_f}(\mat{Z}^{(n)}) = \prod_{k=1,k \ne n}^N \ell_{i_{k}^{(j_f)}}(\mat{A}^{(k)})$. Moreover, with the above $\vect{idx}$, we can find $\mat{Z}^{(n)}_{\mc{F}_n}$ and $ \mat{X}_{(n)}^{\mc{F}_n}$ using \Cref{alg:skr}.

The procedure for the case on the Euclidean-based probability distribution is similar except that the above $\vect{p}_k$ is replaced by $\vect{p}_{k}$ in \Cref{lem:KRP_sample_euc} and the above $p_{j_f}$ is replaced by $p_{j_f} = \prod_{k=1,k \ne n}^N \frac{\| \mat{A}^{(k)} (i_{k}^{(j_f)}, :) \|^2_2}{\| \mat{A}^{(k)} \|_F^2}$. 

Based on the above discussions, we have the following algorithmic framework, i.e., \Cref{alg:KRPSamp1}.

\begin{algorithm}
	\caption{SKRP-ST with importance sampling (SKRP-ST-I)}
	\label{alg:KRPSamp1}
	\begin{algorithmic}[1]\footnotesize
		\Function{$[\mat{Z}^{(n)}_{\mc{F}_n}, \mat{X}_{(n)}^{\mc{F}_n}, \vect{p}_{\mc{F}_n}]$ = SKRP-ST-I}{$n$, $|\mc{F}_n|$, $\{\vect{p}_k\}_{k=1, k \ne n}^{N}, \{\mat{A}^{(k)}\}_{k=1, k \ne n}^{N}$}
		\For{$k=1,\cdots,n-1,n+1,\cdots,N$}
		\State $\vect{idx}(:, k) =$ \textsc{Randsample}($I_k, |\mc{F}_n|, true, \vect{p}_k$) 
		\EndFor
		\State$[\mat{Z}^{(n)}_{\mc{F}_n}, \mat{X}_{(n)}^{\mc{F}_n}]$ = \textsc{SKRP-ST}($n$, $\vect{idx}$, $\{\mat{A}^{(k)}\}_{k=1, k \ne n}^{N}$)
		\State $\vect{p}_{\mc{F}_n} \leftarrow$ \eqref{eq:block_lev_prob} with leveraged/Euclidean-based probability distribution
		\State \Return $\mat{Z}^{(n)}_{\mc{F}_n}, \mat{X}_{(n)}^{\mc{F}_n}, \vect{p}_{\mc{F}_n}$
		\EndFunction
	\end{algorithmic}
\end{algorithm}

\begin{remark}
	The two importance sampling 
	probability distributions  considered above are empirical. 
	In \Cref{thm:optimal_prob} below, we will present a theoretical optimal one. 
	In addition,
	the mini-batches here are constructed by sampling rows first and then combining them into blocks. We can also obtain them by blocking the rows first and then sampling blocks. 
\end{remark}

\subsection{Proposed method}
\label{sec:proposed_algorithm}
We first give the 
stochastic gradient $\mat{G}^{(n)}_{(t)}$, corresponding to the sampling strategies described in \Cref{sec:efficient_sampling},

\begin{align} 
\label{eq:gradient1}
\mat{G}^{(n)}_{(t)} = \frac{1}{|\mc{F}_n| J_n} \left(\mat{A}^{(n)}_{(t)} (\mat{D}(\mat{Z}^{(n)}_{\mc{F}_n})^\intercal ) \mat{Z}^{(n)}_{\mc{F}_n} - \mat{D} \mat{X}_{(n)}^{\mc{F}_n} \mat{Z}^{(n)}_{\mc{F}_n} \right),
\end{align}
where $\mat{D}={\diag}[\frac{1}{p_{j_1}}, \cdots, \frac{1}{p_{j_{|\mc{F}_n|}}}]$
is from \eqref{eq:block_lev_prob}.
Then, the latent factor matrices can be updated by 
\begin{equation}
\label{eq:proposed}
\mat{A}^{(n)}_{(t+1)} \leftarrow \mat{A}^{(n)}_{(t)} - \alpha_{t}\mat{G}^{(n)}_{(t)},~~n=1,\cdots,N.
\end{equation}



Therefore, we can propose our method in \Cref{alg:BR-SGD}. Like \cite{fu2020BlockRandomizedStochastic}, we call it {\it \textbf{B}lock-\textbf{Ra}ndomized \textbf{W}eighted \textbf{S}GD for \textbf{CPD}} (\texttt{BrawsCPD}). 
Similarly, we call the \texttt{BrawsCPD} with the two sampling strategies in \Cref{sec:efficient_sampling} the
\texttt{LBrawsCPD} and \texttt{EBrawsCPD}, respectively.

\begin{algorithm}
	\caption{\texttt{BrawsCPD} 
	}
	\label{alg:BR-SGD}
	\begin{algorithmic}[1]\footnotesize
		\Function{$\{\mat{A}^{(n)}\}_{n=1}^N$= BrawsCPD}{$\tensor{X}, R, |\mc{F}_n|, \{ \mat{A}^{(n)}_{(0)}\}_{n=1}^{N}, \{\alpha_{t}\}$} 
		
		
		
		\State $t\leftarrow 0$
		\Repeat
		\State Uniformly sample $n$ from $\{1,\cdots,N\}$
		\State $[\mat{Z}^{(n)}_{\mc{F}_n}, \mat{X}_{(n)}^{\mc{F}_n}, \vect{p}_{\mc{F}_n}]$ = \textsc{SKRP-ST-I}($n$, $|\mc{F}_n|$, $\{\vect{p}_k\}_{k=1, k \ne n}^{N}$, $\{\mat{A}^{(N)}\}_{k=1, k \ne n}^{N}$)
		\State Form the stochastic gradient $\mat{G}^{(n)}_{(t)} \leftarrow$ \eqref{eq:gradient1} 
		\State Update $\mat{A}^{(n)}_{(t+1)} \leftarrow \eqref{eq:proposed}$, $\mat{A}^{(n')}_{(t+1)} \leftarrow \mat{A}^{(n')}_{(t)}$ for $n'\neq n$
		\State $t\leftarrow t+1$
		\Until{some stopping criterion is reached}
		\State \Return $\{\mat{A}^{(n)}\}_{n=1}^N$
		\EndFunction
	\end{algorithmic}
\end{algorithm}

To avoid running the step size schedule, similar to \cite{fu2020BlockRandomizedStochastic}, we also give an adaptive step size scheme with the following updating rule:
\begin{subequations}
	\begin{align}
	\label{eq:etaada} [\bm{\eta}_{(t)}^{(n)}]_{i,r} &\leftarrow \frac{\eta}{\left(b + \sum_{t'=1}^{t}[\mat{G}^{(n)}_{(t)}]_{i,r}^2\right)^{1/2}},~~i \in [I_n],~r \in [R], \\
	\label{eq:Aada} \mat{A}^{(n)}_{(t+1)}& \leftarrow \mat{A}^{(n)}_{(t)} - \bm{\eta}_{(t)}^{(n)} \circledast \mat{G}^{(n)}_{(t)},\\
	\mat{A}^{(n')}_{(t+1)}&\leftarrow \mat{A}^{(n')}_{(t)},~~n' \neq n,
	\end{align}
\end{subequations}
where $\eta,b > 0$. 
Here, $b>0$ is introduced to prevent division by zero. In practice, setting $b=0$ does not hurt the performance. 
We summary the corresponding algorithm in \Cref{alg:AdawsCPD}, which is named as \texttt{AdawsCPD}. Meanwhile, we name the \texttt{AdawsCPD} with the two sampling strategies in \Cref{sec:efficient_sampling} the \texttt{LAdawsCPD} and \texttt{EAdawsCPD}, respectively.

\begin{algorithm}
	\caption{\texttt{AdawsCPD} 
	}
	\label{alg:AdawsCPD}
	\begin{algorithmic}[1]\footnotesize
		\Function{$\{\mat{A}^{(n)}\}_{n=1}^N$= AdawsCPD}{$\tensor{X}, R, {|\mc{F}_n|}, \{ \mat{A}^{(n)}_{(0)} \}_{n=1}^{N}$} 
		
		\State $t\leftarrow 0$
		\Repeat
		\State Uniformly sample $n$ from $\{1,\cdots,N\}$
		\State $[\mat{Z}^{(n)}_{\mc{F}_n}, \mat{X}_{(n)}^{\mc{F}_n}, \vect{p}_{\mc{F}_n}]$ = \textsc{SKRP-ST-I}($n$, $|\mc{F}_n|$, $\{\vect{p}_k\}_{k=1, k \ne n}^{N}$, $\{\mat{A}^{(N)}\}_{k=1, k \ne n}^{N}$)
		\State Form the stochastic gradient $\mat{G}^{(n)}_{(t)} \leftarrow$ \eqref{eq:gradient1}
		\State Determine the step size $\bm{\eta}_{(t)}^{(n)} \leftarrow$ \eqref{eq:etaada}
		\State Update $\mat{A}^{(n)}_{(t+1)} \leftarrow$ \eqref{eq:Aada}, $\mat{A}^{(n')}_{(t+1)} \leftarrow \mat{A}^{(n')}_{(t)}$ for $n'\neq n$
		\State $t\leftarrow t+1$
		\Until{some stopping criterion is reached}
		\State \Return $\{\mat{A}^{(n)}\}_{n=1}^N$
		\EndFunction
	\end{algorithmic}
\end{algorithm}

\section{Theoretical Properties}
\label{sec:Convergence_Properties}

In this section, we provide some theoretical results of the proposed method, which mainly include the unbiasedness of stochastic gradient, the error analysis of method, and the analysis on variance of stochastic gradient. For simplicity, we will often use the shorthand notation $f(\vect{\theta})$ to denote $f(\mat{A}^{(1)},\cdots,\mat{A}^{(N)}) $,
where $\vect{\theta} =[\mat{A}^{(1)}, \cdots, \mat{A}^{(N)}]$. 
Thus, considering the definition of Frobenius norm, \eqref{eq:certain_criterion} can be rewritten as 
\begin{equation}
\label{eq:sgd_cri2}
\min_{\{\mat{A}^{(n)}\in \bb{R}^{I_n \times R}\}_{n=1}^{N}}~(1/J_n) ~ \sum_{i} f_{i}\left(\vect{\theta}\right),
\end{equation}
where 
$f_{i}(\vect{\theta})=\frac{J_n}{2} \| \mat{Z}^{(n)}(i, :) \mat{A}^\intercal - \mat{X}_{(n)}^\intercal (i, :)\|^2_F$.

\subsection{Unbiasedness of stochastic gradient}\label{sec:Unbias}
For convenience, we define $\xi_{(t)} \in \{1,\cdots,N\}$ and $\zeta_{(t)} \subseteq \{1,\cdots,J_{\xi_{(t)}} \}$ as the random variables (r.v.s) responsible for selecting the mode and fibers in the $t$-th iteration, respectively.

\begin{theorem}[Unbiased Gradient]
	\label{thm:unbias}
	Denote $\mc{B}_{(t)}$ as the filtration generated by the r.v.s $$\{\xi_{(0)},\zeta_{(0)}, \xi_{(1)},\zeta_{(1)}, \cdots,\xi_{(t-1)},\zeta_{(t-1)} \}$$ 
	such that the $t$-th iteration $\vect{\theta}_{(t)}$ is determined conditioned on $\mc{B}_{(t)}$. Then the stochastic gradient in \eqref{eq:gradient1} is the unbiased estimate of the full gradient with respect to (w.r.t.) $\mat{A}^{(\xi_{(t)})}$, i.e.,
	\begin{equation} 
	\label{eq:factunbias}
	\bb{E}_{\zeta_{(t)}}\left[ \mat{G}^{(\xi_{(t)})}_{(t)}~|~\mc{B}_{(t)},\xi_{(t)} \right] = \nabla_{\mat{A}^{(\xi_{(t)})}} f({\vect{\theta}}_{(t)}).
	\end{equation}
\end{theorem}

\begin{proof}
    Note that the stochastic gradient in \eqref{eq:gradient1} can be rewritten as
    \begin{align*} 
    \mat{G}^{(n)}_{(t)} 
    &= \frac{1}{|\mc{F}_n| J_n} \sum_{f=1}^{|\mc{F}_n|} \frac{1}{p_{j_f}} \nabla_{\mat{A}^{(n)}} f_{j_f}(\vect{\theta}_{(t)}).
    \end{align*}
	Hence, we have
	\begin{align*}
	\bb{E}_{\zeta_{(t)}}\left[ \mat{G}^{(\xi_{(t)})}_{(t)}~|~\mc{B}_{(t)},\xi_{(t)} \right] &= \frac{1}{|\mc{F}_n| J_n} \sum_{f=1}^{|\mc{F}_n|} \bb{E} \left(\frac{1}{p_{j_f}}\nabla_{\mat{A}^{(\xi_{(t)})}} f_{j_f}(\vect{\theta}_{(t)}) \right) \\
	&= \frac{1}{|\mc{F}_n| J_n} \sum_{f=1}^{|\mc{F}_n|} \sum_{j_f \in [J_n]} \left(\frac{1}{p_{j_f}} p_{j_f} \nabla_{\mat{A}^{(\xi_{(t)})}} f_{j_f}(\vect{\theta}_{(t)})\right ) \\
	&= \frac{1}{J_n} \sum_{j_f \in [J_n]} \nabla_{\mat{A}^{(\xi_{(t)})}} f_{j_f}(\vect{\theta}_{(t)})
	= \nabla_{\mat{A}^{(\xi_{(t)})}} f(\vect{\theta}_{(t)}).
	\end{align*} 
	So, 
	the desired result holds. 
\end{proof}

\begin{remark}
	Based on the unbiasedness of the stochastic gradient given in \eqref{eq:gradient1}, under some assumptions presented in \cite{fu2020BlockRandomizedStochastic}, along the same line of proofs of Propositions 1 and 3 in \cite{fu2020BlockRandomizedStochastic}, we can show that the solution sequence produced by \texttt{BrawsCPD} satisfies 
	\begin{equation*}
	\liminf_{t \rightarrow \infty} \bb{E} \left[ \| \nabla f(\vect{\theta}_{(t)}) \|_F^2 \right] = 0,
	\end{equation*}
	and the solution sequence produced by \texttt{AdawsCPD} satisfies
	\begin{equation*}
	{\sf Pr}\left( \liminf_{t \rightarrow \infty} \| \nabla f(\vect{\theta}_{(t)}) \|_F^2 = 0 \right) = 1.
	\end{equation*}
	That is, these sequences converge to the corresponding stationary points.
\end{remark}

\subsection{Error analysis}
To investigate the 
error analysis of the proposed method, we need some preparations. They are presented in the following two lemmas, which are mainly from the facts that the objective function $f(\vect{\theta})$ is \emph{block multiconvex} \cite{xu2013BlockCoordinate} and has \emph{block-wise Lipschitz continuous gradient} \cite{fu2020BlockRandomizedStochastic}. This is because $f(\vect{\theta})$ w.r.t. $\mat{A}^{(n)}$ is a plain least squares fitting criterion.

\begin{lemma}[Block Multiconvex \cite{xu2013BlockCoordinate}]
	\label{lem:strong_convex}
	Suppose that $\mat{Z}^{(n)}$ has full column rank. Then, for $\hat{\vect{\theta}}$, $\bar{\vect{\theta}}$ and the mode $n \in \{1,\cdots, N \}$, there exists a constant $\mu^{(n)}$ such that 
	\begin{equation*}
	f(\hat{\vect{\theta}}) \geq f(\bar{\vect{\theta}}) + \left< \nabla_{\mat{A}^{(n)}} f(\bar{\vect{\theta}}), \hat{\mat{A}}^{(n)} - \bar{\mat{A}}^{(n)} \right> + \frac{\mu^{(n)}}{2} \| \hat{\mat{A}}^{(n)} - \bar{\mat{A}}^{(n)} \|_F^2,
	\end{equation*}
	where $\hat{\vect{\theta}}$ and $\bar{\vect{\theta}}$ are the same except the $n$-th factor, i.e., $\hat{\mat{A}}^{(i)}= \bar{\mat{A}}^{(i)}$ for $i \neq n$. 
\end{lemma}

\begin{lemma}[Block-wise Lipschitz Continuous Gradient \cite{fu2020BlockRandomizedStochastic}]
	\label{lem:lipschitz}
	For $\hat{\vect{\theta}}$, $\bar{\vect{\theta}}$ and the mode $n \in \{1,\cdots, N \}$, there exists a constant $L^{(n)}$ such that 
	\begin{equation*}
	f(\hat{\vect{\theta}}) \leq f(\bar{\vect{\theta}}) + \left< \nabla_{\mat{A}^{(n)}} f(\bar{\vect{\theta}}), \hat{\mat{A}}^{(n)} - \bar{\mat{A}}^{(n)} \right> + \frac{L^{(n)}}{2} \| \hat{\mat{A}}^{(n)} - \bar{\mat{A}}^{(n)} \|_F^2,
	\end{equation*}
	where $\hat{\vect{\theta}}$ and $\bar{\vect{\theta}}$ are the same except the $n$-th factor, i.e., $\hat{\mat{A}}^{(i)}= \bar{\mat{A}}^{(i)}$ for $i \neq n$. 
\end{lemma}

\begin{remark}
	It is well known that  $\mu^{(n)}$ and $L^{(n)}$ can be chosen as $\lambda_{\min}((\mat{Z}^{(n)})^\intercal \mat{Z}^{(n)})$ and $\lambda_{\max}((\mat{Z}^{(n)})^\intercal \mat{Z}^{(n)})$, respectively. 
	In addition, the assumption of \Cref{lem:strong_convex} can always be satisfied in our problem, and 
	by the monotonicity of strongly convex function, we have 
	\begin{equation}
	\label{eq:convex}
	\left< \nabla_{\mat{A}^{(n)}} f(\hat{\vect{\theta}}) - \nabla_{\mat{A}^{(n)}} f(\bar{\vect{\theta}}), \hat{\mat{A}}^{(n)} - \bar{\mat{A}}^{(n)} \right> \geq \mu^{(n)} \| \hat{\mat{A}}^{(n)} - \bar{\mat{A}}^{(n)} \|_F^2.
	\end{equation}
\end{remark}

\begin{theorem}[Main Theorem]
	\label{thm:conv_unbra}
	Let $\mat{A}_{*}^{(n)}$ for $n=1,\cdots,N$ be the factor matrices derived by the regular CP-ALS, $\tensor{X}_*=\llbracket \mat{A}_*^{(1)},\mat{A}_*^{(2)}, \cdots,\mat{A}_*^{(N)} \rrbracket$ and $\tensor{X}_{true}$ be the true data tensor. 
	Suppose that the updates $\mat{A}^{(n)}_{(t)}$ are bounded for all $n, t$, and the step size $\alpha_{t}$ is fixed as $\alpha$ satisfying $0 < \alpha < \frac{1}{2 \mu}$ with $\mu$ being a constant. Then, after $T$ iterations, the solution sequence produced by \texttt{BrawsCPD} satisfies:
	\begin{align}
	\label{eq:conv}
	\bb{E} \left[ \| \llbracket \mat{A}^{(1)}_{(T)},\mat{A}^{(2)}_{(T)}, \cdots,\mat{A}^{(N)}_{(T)} \rrbracket - \tensor{X}_{true} \|_F^2 \right] 
	&\nonumber \leq \frac{L}{4} \left[ (1-2 \alpha \mu)^T \Delta_{0}^2 + \frac{\alpha M}{2 \mu} \right] \\ 
	&\qquad + \frac{1}{2} \| \tensor{X}_* - \tensor{X}_{true} \|_F^2,
	\end{align}
	where $L$ is a constant, $\Delta_t = \| \mat{A}_{(t)}^{(\xi_{(t)})} - \mat{A}_{*}^{(\xi_{(t)})} \|_F$, and 
	$$M =\max\left\{ \bb{E}_{\zeta_{(t)}}\left[ \| \mat{G}^{(\xi_{(t)})}_{(t)} \|_F^2 ~|~ \mc{B}_{(t)},\xi_{(t)} \right]\right\}$$ 
	with $\mc{B}_{(t)}$, $\xi_{(t)}$ and $\zeta_{(t)}$ being defined as in \Cref{sec:Unbias}.
\end{theorem}

\begin{proof}
	On the one hand, we have
	\begin{equation*}
	\bb{E} \left[ \| \llbracket \mat{A}^{(1)}_{(T)},\mat{A}^{(2)}_{(T)}, \cdots,\mat{A}^{(N)}_{(T)} \rrbracket - \tensor{X}_{true} \|_F^2 \right] = \frac{1}{2} \bb{E} \left[f(\vect{\theta}_{(T)})\right]
	\end{equation*}
	and
	\begin{equation*}
	\bb{E} \left[f(\vect{\theta}_{(T)})\right] = \bb{E} \left[f(\vect{\theta}_{(T)}) - f(\vect{\theta}_{(*)})\right] + \| \tensor{X}_* - \tensor{X}_{true} \|_F^2. 
	\end{equation*}
	On the other hand, using \Cref{lem:lipschitz}, we get
	\begin{equation*}
	f(\vect{\theta}_{(T)}) - f(\vect{\theta}_{(*)}) 
	\leq \left< \nabla_{\mat{A}^{(\xi_{(t)})}} f(\vect{\theta}_{*}), \mat{A}_{(T)}^{(\xi_{(t)})} - \mat{A}_{*}^{(\xi_{(t)})} \right> + \frac{L}{2} \|  \mat{A}_{(T)}^{(\xi_{(t)})} - \mat{A}_{*}^{(\xi_{(t)})} \|_F^2,
	\end{equation*}
	where $L = \max_{t=0,\cdots,\infty} L_{(t)}^{(\xi_{(t)})} < \infty$, which together with
	$\nabla_{\mat{A}^{(\xi_{(t)}})} f(\vect{\theta}_{*}) = 0$ implies 
	\begin{equation*}
	\bb{E} \left[ f(\vect{\theta}_{(T)}) - f(\vect{\theta}_{(*)}) \right] \leq \frac{L}{2} \bb{E} \left[ \Delta_{T}^2 \right].
	\end{equation*}
	As a result,
	\begin{equation*}
	\bb{E} \left[ \| \llbracket \mat{A}^{(1)}_{(T)},\mat{A}^{(2)}_{(T)}, \cdots,\mat{A}^{(N)}_{(T)} \rrbracket - \tensor{X}_{true} \|_F^2 \right]  \leq \frac{L}{4} \bb{E} \left[ \Delta_{T}^2 \right]+ \frac{1}{2}\| \tensor{X}_* - \tensor{X}_{true} \|_F^2. 
	\end{equation*}
	Thus, to get the desired result, it suffices to estimate $\bb{E} \left[ \Delta_{T}^2 \right]$. 
	
	We begin from the expression of $\Delta_{t+1}^2$:
	\begin{align*}
	\Delta_{t+1}^2 &= \| \mat{A}_{(t+1)}^{(\xi_{(t)})} - \mat{A}_{*}^{(\xi_{(t)})} \|_F^2 
	= \| \mat{A}_{(t)}^{(\xi_{(t)})} - \alpha_{t} \mat{G}^{(\xi_{(t)})}_{(t)} - \mat{A}_{*}^{(\xi_{(t)})} \|_F^2 \\
	&= \| \mat{A}_{(t)}^{(\xi_{(t)})} - \mat{A}_{*}^{(\xi_{(t)})} \|_F^2 - 2 \alpha_{t} \left< \mat{G}^{(\xi_{(t)})}_{(t)}, \mat{A}_{(t)}^{(\xi_{(t)})} - \mat{A}_{*}^{(\xi_{(t)})}\right> + \alpha_{t}^2 \| \mat{G}^{(\xi_{(t)})}_{(t)} \|_F^2,
	\end{align*}
	which is from the update formula of \texttt{BrawsCPD} and some algebra.
	Taking the expectation conditioned on the filtration $\mc{B}_{(t)}$ and the chosen mode index $\xi_{(t)}$ gives
	\begin{align*}
	&\bb{E}_{\zeta_{(t)}} \left[ \Delta_{t+1}^2 ~|~ \mc{B}_{(t)},\xi_{(t)} \right] \\
	&\qquad = \bb{E}_{\zeta_{(t)}} \left[ \Delta_t^2 ~|~ \mc{B}_{(t)},\xi_{(t)} \right] - 2 \alpha_{t} \bb{E}_{\zeta_{(t)}} \left[ \left< \mat{G}^{(\xi_{(t)})}_{(t)}, \mat{A}_{(t)}^{(\xi_{(t)})} - \mat{A}_{*}^{(\xi_{(t)})}\right> ~|~ \mc{B}_{(t)},\xi_{(t)} \right] \\
	&\qquad \qquad + \alpha_{t}^2 \bb{E}_{\zeta_{(t)}} \left[ \| \mat{G}^{(\xi_{(t)})}_{(t)} \|_F^2 ~|~ \mc{B}_{(t)},\xi_{(t)} \right] \\
	&\qquad \leq \bb{E}_{\zeta_{(t)}} \left[ \Delta_t^2 ~|~ \mc{B}_{(t)},\xi_{(t)} \right] - 2 \alpha_t \left< \nabla_{\mat{A}^{(\xi_{(t)})}} f(\vect{\theta}_{(t)}), \mat{A}_{(t)}^{(\xi_{(t)})} - \mat{A}_{*}^{(\xi_{(t)})} \right> + \alpha_t^2 M \\
	&\qquad \leq \bb{E}_{\zeta_{(t)}} \left[ \Delta_t^2 ~|~ \mc{B}_{(t)},\xi_{(t)} \right] - 2 \alpha_t \mu \Delta_t^2 + \alpha_t^2 M,
	\end{align*}
	where the first inequality follows from \Cref{thm:unbias} and the definition of $M$, the second inequality is from \eqref{eq:convex}, and $\mu= \max_{t=0,\cdots,\infty} \mu_{(t)}^{(\xi_{(t)})} < \infty$.
	Now, taking the expectation w.r.t $\xi_{(t)}$, we get
	\begin{align*}
	\bb{E}_{\zeta_{(t)},\xi_{(t)}} \left[ \Delta_{t+1}^2 ~|~ \mc{B}_{(t)} \right] 
	&\leq \bb{E}_{\zeta_{(t)},\xi_{(t)}} \left[ \Delta_{t}^2 ~|~ \mc{B}_{(t)} \right] - 2 \alpha_t \mu \bb{E}_{\zeta_{(t)},\xi_{(t)}} \left[ \Delta_{t}^2 ~|~ \mc{B}_{(t)} \right] + \alpha_t^2 M \\
	&= (1-2 \alpha_t \mu) \bb{E}_{\zeta_{(t)},\xi_{(t)}} \left[ \Delta_{t}^2 ~|~ \mc{B}_{(t)} \right] + \alpha_t^2 M.
	\end{align*}
	Further, taking the total expectation w.r.t all random variables in $\mc{B}_{(t)}$ yields
	\begin{equation*}
	\bb{E} \left[ \Delta_{t+1}^2 \right] \leq (1-2 \alpha_t \mu) \bb{E} \left[ \Delta_{t}^2 \right] + \alpha_t^2 M.
	\end{equation*}
	Note that the step size is fixed. Thus, making an induction of $t = T-1$, we obtain
	\begin{equation*}
	\bb{E} \left[ \Delta_{T}^2 \right] \leq (1-2 \alpha \mu)^T \Delta_{0}^2 + \sum_{i=0}^{T-1} (1-2 \alpha \mu)^i \alpha^2 M.
	\end{equation*}
	According to the assumption $0 < 2 \alpha \mu < 1$, we can check that
	\begin{equation*}
	\sum_{i=0}^{T-1} (1-2 \alpha \mu)^i < \sum_{i=0}^{\infty} (1-2 \alpha \mu)^i = \frac{1}{2 \alpha \mu}.
	\end{equation*}
	Therefore,
	\begin{equation*}
	\bb{E} \left[ \Delta_{T}^2 \right] \leq (1-2 \alpha \mu)^T \Delta_{0}^2 + \frac{\alpha M}{2 \mu}.
	\end{equation*}
	Finally, the desired result is derived.
\end{proof}

\begin{remark}
	The assumption on the boundedness of $\mat{A}^{(n)}_{(t)}$ is to guarantee that $L_{(t)}^{(\xi_{(t)})}$ is upper bounded. A simple way to make  $\mat{A}^{(n)}_{(t)}$ bounded is to scale $\mat{A}^{(1)}, \mat{A}^{(2)}, \cdots, \mat{A}^{(N)}$ after each iteration.
\end{remark}

Rewriting \eqref{eq:conv} gives
\begin{align*}
&\bb{E} \left[ \| \llbracket \mat{A}^{(1)}_{(T)},\mat{A}^{(2)}_{(T)}, \cdots,\mat{A}^{(N)}_{(T)} \rrbracket - \tensor{X}_{true} \|_F^2 \right] -\frac{1}{2} \| \tensor{X}_* - \tensor{X}_{true} \|_F^2\\
&\leq \frac{L}{4} \left[ (1-2 \alpha \mu)^T \Delta_{0}^2 + \frac{\alpha M}{2 \mu} \right],
\end{align*}
which describes the difference between the errors caused by \texttt{BrawsCPD} and the regular CP-ALS. When the step size $\alpha_{t}$ is fixed, the difference cannot be eliminated. This is because the term $\frac{L}{4} \frac{\alpha M}{2 \mu}$ in the right side of the above expression is independent of the iteration number $T$. In the following, we consider the changing step size.

\begin{theorem}
	\label{thm:de_step_bra}
	In the setting of \Cref{thm:conv_unbra}, take the decreasing step size as
	\begin{equation*}
	\alpha_t = \frac{\beta}{t+\gamma},
	\end{equation*}
	where  $\beta > \frac{1}{2 \mu}$ and $\gamma > 0$ such that $\alpha_1 \leq \frac{1}{2 \mu}$. Then, for $\forall t \geq 1$, we have
	\begin{equation}
	\label{eq:de_conv}
	\bb{E} \left[ \| \llbracket \mat{A}^{(1)}_{(t)},\mat{A}^{(2)}_{(t)}, \cdots,\mat{A}^{(N)}_{(t)} \rrbracket - \tensor{X}_{true} \|_F^2 \right]-\frac{1}{2} \| \tensor{X}_* - \tensor{X}_{true} \|_F^2 
	\leq \frac{L}{4} \frac{v}{\gamma+t},
	\end{equation}
	where $v = \max \left\{ \frac{\beta^2 M}{2 \beta \mu - 1}, \gamma \Delta_0^2 \right\}$.
\end{theorem}

\begin{proof}
	In the proof of \Cref{thm:conv_unbra}, we have shown that 
	\begin{equation*}
	\bb{E} \left[ \Delta_{t+1}^2 \right] \leq (1-2 \alpha_t \mu) \bb{E} \left[ \Delta_{t}^2 \right] + \alpha_t^2 M.
	\end{equation*}
	In the following, we prove $\bb{E} \left[ \Delta_{t}^2 \right] 
	\leq \frac{v}{\gamma+t} $
	by mathematical induction. When $t = 0$, the formula is established by the definition of $v$.
	Now, we assume that this formula holds for $t$. 
	For the sake of notation, we set $\hat{t} = \gamma + t$. Then $\alpha_t = \frac{\beta}{\hat{t}}$ and $\bb{E} \left[ \Delta_{t}^2 \right] 
	\leq \frac{v}{\hat{t}} $. Thus, 
	\begin{align*}
	\bb{E} \left[ \Delta_{t+1}^2 \right] &\leq (1-2 \alpha_t \mu) \frac{v}{\hat{t}} + \alpha_t^2 M 
	= (1 - \frac{2 \beta \mu}{\hat{t}}) \frac{v}{\hat{t}} + \frac{\beta^2 M}{\hat{t}^2} \\
	&= \frac{\hat{t}-1}{\hat{t}^2}v - \frac{2 \beta \mu - 1}{\hat{t}^2}v + \frac{\beta^2 M}{\hat{t}^2} 
	\leq \frac{v}{\hat{t}+1},
	\end{align*}
	where the last inequality follows from the definition of $v$. Therefore, the formula also holds for $t+1$, and hence the desired result in \eqref{eq:de_conv} is derived.
\end{proof}

\begin{remark}
	From \eqref{eq:de_conv}, we can see that the difference between the errors caused by \texttt{BrawsCPD} and the
	regular CP-ALS can approach zero as the iteration number $T$ grows. More specifically, 
	to ensure the difference be smaller than some target error $\epsilon$, it suffices to set $T \geq \frac{L v}{4 \epsilon} + \gamma$. Note that for the case on fixed step size, to get the above aim\footnote{In fact, this aim cannot be truly achieved since $8 \mu \epsilon - \alpha L M$ may be smaller than zero when $\epsilon$ is small enough. It again shows that, for the fixed step size case, the difference between the errors caused by \texttt{BrawsCPD} and the regular CP-ALS cannot be eliminated.}, the iteration number $T$ should satisfy  
	\begin{equation}
	\label{eq:iters}
	T \geq \frac{1}{2 \alpha \mu} \ln \left( \frac{2 \mu \Delta_{0}^2 L}{8 \mu \epsilon - \alpha L M} \right).    
	\end{equation} 
\end{remark}

\begin{remark}
	For \texttt{AdawsCPD}, its step size is also changing. Similar to the proofs of \Cref{thm:conv_unbra,thm:de_step_bra}, one may present the error analysis for this algorithm. However, the process is very tedious and we cannot find an elegant result at present. So we omit it here and leave it for future research.
\end{remark}

\subsection{Variance of stochastic gradient}
Upon closer examination of \Cref{thm:conv_unbra} and \eqref{eq:iters}, we can see that the error bound and iteration number have a close relationship with $M$. Further, note that $M \leq \max \left\{ \bb{E}_{\zeta_{(t)}}\left[ \| \mat{G}^{(\xi_{(t)})}_{(t)} - \nabla_{\mat{A}^{(\xi_{(t)})}} f(\bm{\theta}_{(t)}) \|_F^2 ~|~\mc{B}_{(t)},\xi_{(t)} \right] \right\} + L^2 \Delta_t^2$. 
So, in the following, we present the specific expression of the variance. 

\begin{theorem}
\label{thm:var}
	In the setting of \Cref{thm:unbias},
	suppose that $\vect{p} \in \bb{R}^{J_n}$ is any probability distribution proposed in \Cref{sec:efficient_sampling}, and $\mat{R}^{(n)}_{(t)} = \mat{A}^{(n)}_{(t)} (\mat{Z}^{(n)}_{(t)})^\intercal - \mat{X}_{(n)}$. Then 
	\begin{align}
	\label{eq:var_p}
	\nonumber &\bb{E}_{\zeta_{(t)}}\left[ \| \mat{G}^{(\xi_{(t)})}_{(t)} - \nabla_{\mat{A}^{(\xi_{(t)})}} f(\vect{\theta}_{(t)}) \|_F^2 ~|~\mc{B}_{(t)},\xi_{(t)} \right] \\
	&= \frac{1}{|\mc{F}_n|} \sum_{j_f=1}^{J_n} \frac{1}{p_{j_f}} \| \mat{R}^{(\xi_{(t)})}_{(t)}(:,j_f) \|_2^2 \| \mat{Z}^{(\xi_{(t)})}_{(t)}(j_f,:) \|_2^2 - \frac{1}{|\mc{F}_n|} \| \nabla_{\mat{A}^{(\xi_{(t)})}} f(\vect{\theta}_{(t)}) \|_F^2.
	\end{align}
	
\end{theorem}

\begin{proof}
	
	Define
	\begin{equation*}
	\Theta_f = \frac{1}{|\mc{F}_n| J_n p_{j_f}} \nabla_{\mat{A}^{(\xi_{(t)})}}  f_{j_f}(\vect{\theta}),
	\end{equation*}
	where $f = 1,\cdots, |\mc{F}_n|$.
	Thus,
	\begin{equation*}
	\bb{E}[\Theta_f] = \sum_{j_f=1}^{J_n} \frac{1}{|\mc{F}_n| J_n p_{j_f}} p_{j_f} \nabla_{\mat{A}^{(\xi_{(t)})}}  f_{j_f}(\vect{\theta}) = \frac{1}{|\mc{F}_n|} \nabla_{\mat{A}^{(\xi_{(t)})}}  f(\vect{\theta}),
	\end{equation*}
	and
	\begin{align*}
	\bb{V}[(\Theta_f)_{i,r}] 
	&= \bb{E}[(\Theta_f)_{i,r}^2] - \bb{E}^2[(\Theta_f)_{i,r}] \\
	&= \frac{1}{|\mc{F}_n|^2 J_n^2} \sum_{j_f=1}^{J_n} \left( \frac{\left[ \nabla_{\mat{A}^{(\xi_{(t)})}} f_{j_f}(\vect{\theta}) \right]_{i,r}^2}{p_{j_f}} \right) - \frac{1}{|\mc{F}_n|^2} \left[ \nabla_{\mat{A}^{(\xi_{(t)})}} f(\vect{\theta}) \right]_{i,r}^2.
	\end{align*}
	Since $\bb{V} \left[ \left( \mat{G}^{(\xi_{(t)})}_{(t)} \right)_{i,r} \right] = \sum_{f=1}^{F} \bb{V}[(\Theta_f)_{i,r}]$, we have
	\begin{equation*}
	\bb{V} \left[ \left( \mat{G}^{(\xi_{(t)})}_{(t)} \right)_{i,r} \right] = \frac{1}{|\mc{F}_n| J_n^2} \sum_{j_f=1}^{J_n} \left( \frac{\left[ \nabla_{\mat{A}^{(\xi_{(t)})}} f_{j_f}(\vect{\theta}) \right]_{i,r}^2}{p_{j_f}} \right) - \frac{1}{|\mc{F}_n|} \left[ \nabla_{\mat{A}^{(\xi_{(t)})}} f(\vect{\theta}) \right]_{i,r}^2.
	\end{equation*}
	
	On the other hand, 
	\begin{align*}
	&\bb{E}_{\zeta_{(t)}}\left[ \| \mat{G}^{(\xi_{(t)})}_{(t)} - \nabla_{\mat{A}^{(\xi_{(t)})}} f(\vect{\theta}_{(t)}) \|_F^2 ~|~ \mc{B}_{(t)},\xi_{(t)} \right] \\
	&\qquad = \sum_{i=1}^{I_n} \sum_{r=1}^{R} \bb{E} \left[ \left( \mat{G}^{(\xi_{(t)})}_{(t)} - \nabla_{\mat{A}^{(\xi_{(t)})}} f(\vect{\theta}) \right)_{i,r}^2 \right]  = \sum_{i=1}^{I_n} \sum_{r=1}^{R} \bb{V} \left[ \left( \mat{G}^{(\xi_{(t)})}_{(t)} \right)_{i,r} \right].
	\end{align*}
	Therefore,
	\begin{align*}
	&\bb{E}_{\zeta_{(t)}}\left[ \| \mat{G}^{(\xi_{(t)})}_{(t)} - \nabla_{\mat{A}^{(\xi_{(t)})}} f(\vect{\theta}_{(t)}) \|_F^2 ~|~ \mc{B}_{(t)},\xi_{(t)} \right] \\
	&\qquad = \frac{1}{|\mc{F}_n| J_n^2} \sum_{j_f=1}^{J_n} \frac{1}{p_{j_f}} \sum_{i=1}^{I_n} \sum_{r=1}^{R} \left[ \nabla_{\mat{A}^{(\xi_{(t)})}} f_{j_f}(\vect{\theta}) \right]_{i,r}^2 - \frac{1}{|\mc{F}_n|} \| \nabla_{\mat{A}^{(\xi_{(t)})}} f(\vect{\theta}) \|_F^2 \\
	&\qquad = \frac{1}{|\mc{F}_n|} \sum_{j_f=1}^{J_n} \frac{1}{p_{j_f}} \left( \sum_{i=1}^{I_n} (\mat{R}^{(\xi_{(t)})}_{(t)})_{i,j_f}^2 \right) \left( \sum_{r=1}^{R} (\mat{Z}^{(\xi_{(t)})}_{(t)})_{j_f,r}^2 \right) - \frac{1}{|\mc{F}_n|} \| \nabla_{\mat{A}^{(\xi_{(t)})}} f(\vect{\theta}) \|_F^2 \\
	&\qquad = \frac{1}{|\mc{F}_n|} \sum_{j_f=1}^{J_n} \frac{1}{p_{j_f}} \| \mat{R}^{(\xi_{(t)})}_{(t)}(:,j_f) \|_2^2 \| \mat{Z}^{(\xi_{(t)})}_{(t)}(j_f,:) \|_2^2 - \frac{1}{|\mc{F}_n|} \| \nabla_{\mat{A}^{(\xi_{(t)})}} f(\vect{\theta}) \|_F^2,
	\end{align*} 	
	where the second equality follows from $\nabla_{\mat{A}^{(\xi_{(t)})}} f_{j_f}(\vect{\theta}) = J_n \mat{R}^{(\xi_{(t)})}_{(t)}(:,j_f) \mat{Z}^{(\xi_{(t)})}_{(t)}(j_f,:)$. 
\end{proof}

\begin{remark}
	We list the variances for three specific probability distributions in \Cref{tab:var}.  Since some leverage scores or squared Euclidean norms of the rows of $\mat{A}^{(k)}$ may be very small, the corresponding variances may be larger than the one for uniform sampling. This result is similar to the finding in \cite{ma2015StatisticalPerspective}. 
	Hence, it is difficult to compare these variances in theory. Numerical results in  
	\Cref{sec:Numerical_Results} 
	show that, in most cases, the variances for importance sampling are much smaller than the one for uniform sampling.
	\begin{table}[htbp] 
		\caption{Variances for different probability distributions} 
		\label{tab:var}
		\resizebox{1\linewidth}{!}{
			\begin{tabular}{llll} 
				\toprule
				Probability distributions & $\bb{E}_{\zeta_{(t)}}\left[ \| \mat{G}^{(\xi_{(t)})}_{(t)} - \nabla_{\mat{A}^{(\xi_{(t)})}} f(\mat{\theta}_{(t)}) \|_F^2 ~|~\mc{B}_{(t)},\xi_{(t)} \right]$ \\
				\midrule
				Uniform & $\frac{J_n}{|\mc{F}_n|} \sum_{j_f=1}^{J_n} \| \mat{R}^{(\xi_{(t)})}_{(t)}(:,j_f) \|_2^2 \| \mat{Z}^{(\xi_{(t)})}_{(t)}(j_f,:) \|_2^2 - \frac{1}{|\mc{F}_n|} \| \nabla_{\mat{A}^{(\xi_{(t)})}} f(\vect{\theta}_{(t)}) \|_F^2$ \\
				Leverage-based & $\frac{R^{N-1}}{|\mc{F}_n|} \sum_{j_f=1}^{J_n} \frac{ \| \mat{R}^{(\xi_{(t)})}_{(t)}(:,j_f) \|_2^2 \| \mat{Z}^{(\xi_{(t)})}_{(t)}(j_f,:) \|_2^2}{\prod_{k \ne \xi_{(t)}} \ell_{i_{k}^{(j_f)}}(\mat{A}^{(k)})} - \frac{1}{|\mc{F}_n|} \| \nabla_{\mat{A}^{(\xi_{(t)})}} f(\vect{\theta}_{(t)}) \|_F^2$ \\
				Euclidean-based & $\frac{\prod_{k \ne \xi_{(t)}} \| \mat{A}^{(k)} \|_F^2}{|\mc{F}_n|} \sum_{j_f=1}^{J_n} \frac{\| \mat{R}^{(\xi_{(t)})}_{(t)}(:,j_f) \|_2^2 \| \mat{Z}^{(\xi_{(t)})}_{(t)}(j_f,:) \|_2^2}{\prod_{k \ne \xi_{(t)}} \| \mat{A}^{(k)} (i_{k}^{(j_f)}, :) \|^2_2} - \frac{1}{|\mc{F}_n|} \| \nabla_{\mat{A}^{(\xi_{(t)})}} f(\vect{\theta}_{(t)}) \|_F^2$ \\
				\bottomrule
			\end{tabular}
		}
	\end{table} 
	
	
	
\end{remark}

On the basis of \Cref{thm:var}, in the following, we theoretically give the optimal sampling probability distribution in the sense of minimizing variance.

\begin{theorem}[Optimal Sampling Probability]
	\label{thm:optimal_prob}
	In the setting of \Cref{thm:unbias},
	suppose that $\vect{p} \in \bb{R}^{J_n}$ 
	is any probability distribution 
	and $\mat{R}^{(n)}_{(t)} = \mat{A}^{(n)}_{(t)} (\mat{Z}^{(n)}_{(t)})^\intercal - \mat{X}_{(n)}$. Then if $\vect{p}$ is as 
	\begin{equation}
	\label{eq:optimal_p}
	p_i = \frac{\| \mat{R}^{(\xi_{(t)})}_{(t)}(:,i) \|_2 \| \mat{Z}^{(\xi_{(t)})}_{(t)}(i,:) \|_2}{\sum_{i'=1}^{J_n} \| \mat{R}^{(\xi_{(t)})}_{(t)}(:,i') \|_2 \| \mat{Z}^{(\xi_{(t)})}_{(t)}(i',:) \|_2}, \ i=1,\cdots, J_n,
	\end{equation}
	$\bb{E}_{\zeta_{(t)}}\left[ \| \mat{G}^{(\xi_{(t)})}_{(t)} - \nabla_{\mat{A}^{(\xi_{(t)})}} f(\vect{\theta}_{(t)}) \|_F^2 ~|~\mc{B}_{(t)},\xi_{(t)} \right] $ achieves its minimum as 
	\begin{align}
	\label{eq:opt_prob}
	\frac{1}{|\mc{F}_n|} \left( \sum_{j_f=1}^{J_n} \| \mat{R}^{(\xi_{(t)})}_{(t)}(:,j_f) \|_2 \| \mat{Z}^{(\xi_{(t)})}_{(t)}(j_f,:) \|_2 \right)^2 - \frac{1}{|\mc{F}_n|} \| \nabla_{\mat{A}^{(\xi_{(t)})}} f(\vect{\theta}_{(t)}) \|_F^2.
	\end{align}
\end{theorem}

\begin{proof}
	
	Define a function as
	\begin{equation*}
	f(p_1, \cdots p_{J_n}) = \sum_{j_f=1}^{J_n} \frac{1}{p_{j_f}} \| \mat{R}^{(\xi_{(t)})}_{(t)}(:,j_f) \|_2^2 \| \mat{Z}^{(\xi_{(t)})}_{(t)}(j_f,:) \|_2^2,
	\end{equation*}
	which characterizes the dependence of the variance 
	\begin{equation*}
	\bb{E}_{\zeta_{(t)}}\left[ \| \mat{G}^{(\xi_{(t)})}_{(t)} - \nabla_{\mat{A}^{(\xi_{(t)})}} f(\vect{\theta}_{(t)}) \|_F^2 ~|~\mc{B}_{(t)},\xi_{(t)} \right]
	\end{equation*}
	on the sampling probability distribution $\vect{p}$. To minimize $f$ subject to $\sum_{i=1}^{J_n} p_i = 1$, we introduce the Lagrange multiplier $\lambda$ and define the function
	\begin{equation*}
	g(p_1, \cdots p_{J_n}) = f(p_1, \cdots p_{J_n}) + \lambda \left( \sum_{i=1}^{J_n} p_i - 1 \right).
	\end{equation*}
	Since
	\begin{equation*}
	0 = \frac{\partial g}{\partial p_i} = \frac{-1}{p_i^2} \| \mat{R}^{(\xi_{(t)})}_{(t)}(:,i) \|_2^2 \| \mat{Z}^{(\xi_{(t)})}_{(t)}(i,:) \|_2^2 + \lambda,
	\end{equation*}
	we have
	\begin{equation*}
	p_i = \frac{\| \mat{R}^{(\xi_{(t)})}_{(t)}(:,i) \|_2^2 \| \mat{Z}^{(\xi_{(t)})}_{(t)}(i,:) \|_2^2}{\sqrt{\lambda}} 
	= \frac{\| \mat{R}^{(\xi_{(t)})}_{(t)}(:,i) \|_2^2 \| \mat{Z}^{(\xi_{(t)})}_{(t)}(i,:) \|_2^2}{\sum_{i'=1}^{J_n} \| \mat{R}^{(\xi_{(t)})}_{(t)}(:,i') \|_2^2 \| \mat{Z}^{(\xi_{(t)})}_{(t)}(i',:) \|_2^2},
	\end{equation*}
	where the second equality is from the fact that $\sum_{i=1}^{J_n} p_i = 1$. Further, note that, for the above probabilities, $\frac{\partial^2 g}{\partial p_i^2} > 0$. 
	Hence, the probability distribution $\vect{p}$ in \eqref{eq:optimal_p} minimizes the variance.
	
	On the other hand, substituting $\vect{p}$ in \eqref{eq:optimal_p} into \eqref{eq:var_p} gives \eqref{eq:opt_prob}.
	So, the desired results hold.
\end{proof}

\begin{remark}
	Although the probability distribution in \eqref{eq:optimal_p} is optimal in reducing variance, it is unpractical compared to the ones proposed in \Cref{sec:efficient_sampling}. 
	This is because this probability distribution needs to form the matrix $\mat{R}^{(\xi_{(t)})}_{(t)}$ and compute the norms of its columns in each iteration. So, we don't consider it in numerical experiments.
\end{remark}

\section{Numerical Results}
\label{sec:Numerical_Results}
In this section, we use synthetic and real data to test the effectiveness of our method. 
For maneuverability, we only perform \texttt{AdawsCPD} with two sampling strategies to avoid tuning the step size manually. In addition, we mainly compare our method with \texttt{AdasCPD} from \cite{fu2020BlockRandomizedStochastic} because it mainly build on this method and there are already extensive comparisons between \texttt{AdasCPD} and others methods in \cite{fu2020BlockRandomizedStochastic}. 

\subsection{Environment setup}
\label{sec:experimental setup}
The experiments were carried out  by using the Tensor Toolbox for MATLAB (Version 2018a) \cite{kolda2006TensorToolbox}, and
we used a 2.3 GHz 8-Core Intel Core i9 CPU with 16 GB 2400 MHz DDR4 memory.

\subsection{Synthetic data}
\label{sec:syn_data}
The synthetic tensors of size $I \times I \times I$ are used in our experiments, and, to show the advantages of importance sampling, we use the method from \cite{larsen2020PracticalLeverageBased} to generate the data. 

Specifically,  three $I$-by-$R_{true}$ factor matrices with independent standard Gaussian entries are first generated. Then, the first three columns of each factor matrix are set to be 0. 
Following this,  a data-generating function with two parameters $spread$ and $magnitude$ is applied to those zero columns to make them nonzero, 
where the parameters are used to control the number and size of non-zero elements, 
respectively. 
Finally, for the last factor matrix, we only keep its top 15 rows and set the remaining rows to be 0. Thus, three factor matrices with high leverage scores are created and hence a desired true tensor is generated:
\begin{equation*}
\tensor{X}_{true} = \llbracket \mat{A}^{(1)},\mat{A}^{(2)}, \mat{A}^{(3)} \rrbracket.
\end{equation*}
The observed tensor is obtained by adding suitable noise into the true tensor. That is, 
\begin{equation*}
\tensor{X} = \tensor{X}_{true} + noise \left(\frac{\| \tensor{X}_{true} \|}{\| \tensor{N} \|}\right) \tensor{N},
\end{equation*}
where $\tensor{N} \in \bb{R}^{I \times I \times I}$ is a noise tensor with entries drawn from a standard normal distribution and the parameter $noise$ is the amount of noise. As for the sparse tensor, we will only add the noise to the non-zero entries.

Note that the way of generating data is not unique. Actually, as long as the data has the characteristic of different importance among rows of coefficient matrices of CP-ALS subproblems, our algorithms will have better performance in terms of accuracy or the number of iterations compared with \texttt{AdasCPD}. 

To measure the performance, we compute the relative error $Tol$ after each iteration, 
\begin{equation*}
Tol = \frac{\| \llbracket \hat{\mat{A}}^{(1)},\hat{\mat{A}}^{(2)}, \hat{\mat{A}}^{(3)} \rrbracket - \tensor{X} \|_F^2}{\| \tensor{X} \|_F^2},
\end{equation*}
where $\hat{\mathbf{A}}^{(1)}, \hat{\mat{A}}^{(2)}$, and $ \hat{\mat{A}}^{(3)}$ are the estimated factors, and then record the number of iterations and  running time for the same accuracy. All the results are obtained from 10 trials with tensors generated randomly.

\begin{table}[htbp]
	\caption{Performance of the algorithms with $Tol = 10^{-5}$, $|\mc{F}_n| = 18$, the target rank $R = 10$, $noise = 0$,  and random initialization for different tensors generated by $I \times 10$ factor matrices with different $I$.}
	\label{tab:Ichange}
	\resizebox{1\linewidth}{!}{
		\begin{tabular}{lllllll}
			\hline
			\multicolumn{2}{c}{\multirow{2}{*}{Algorithms}}                                       & $I=100$                                                                & $I=200$                                                                 & $I=300$                                                                 & $I=400$                                                                 & $I=500$                                                                 \\ \cline{3-7} 
			\multicolumn{2}{c}{}                                                                  & \begin{tabular}[c]{@{}l@{}}$spread=15$, \\ $magnitude=24$\end{tabular} & \begin{tabular}[c]{@{}l@{}}$spread=30$, \\  $magnitude=30$\end{tabular} & \begin{tabular}[c]{@{}l@{}}$spread=45$, \\  $magnitude=36$\end{tabular} & \begin{tabular}[c]{@{}l@{}}$spread=60$, \\  $magnitude=42$\end{tabular} & \begin{tabular}[c]{@{}l@{}}$spread=75$, \\  $magnitude=48$\end{tabular} \\ \hline
			\multirow{2}{*}{\texttt{AdasCPD} \cite{fu2020BlockRandomizedStochastic}} & Iterations & 5962.7                                                                 & 4206.3                                                                  & 3105.8                                                                  & 3271.9                                                                  & 4229.5                                                                  \\
			& Seconds    & 23.997029                                                              & 131.46287                                                               & 472.43042                                                               & 1241.3702                                                               & 3224.5938                                                               \\ \hline
			\multirow{2}{*}{\texttt{EAdawsCPD}}                                      & Iterations & 2242.1                                                                 & 572.2                                                                   & 390.3                                                                   & 391.1                                                                   & 488.6                                                                   \\
			& Seconds    & 9.1444463                                                              & 18.08937                                                                & 58.656213                                                               & 150.04822                                                               & 372.96764                                                               \\ \hline
			\multirow{2}{*}{\texttt{LAdawsCPD}}                                      & Iterations & 2394.9                                                                 & 577.8                                                                   & 429.9                                                                   & 483.2                                                                   & 607.3                                                                   \\
			& Seconds    & 10.634388                                                              & 18.25461                                                                & 64.193687                                                               & 184.35716                                                               & 463.1262                                                                \\ \hline
		\end{tabular}
	}
\end{table}

\begin{figure}[htbp]
	\includegraphics[width=0.5\textwidth]{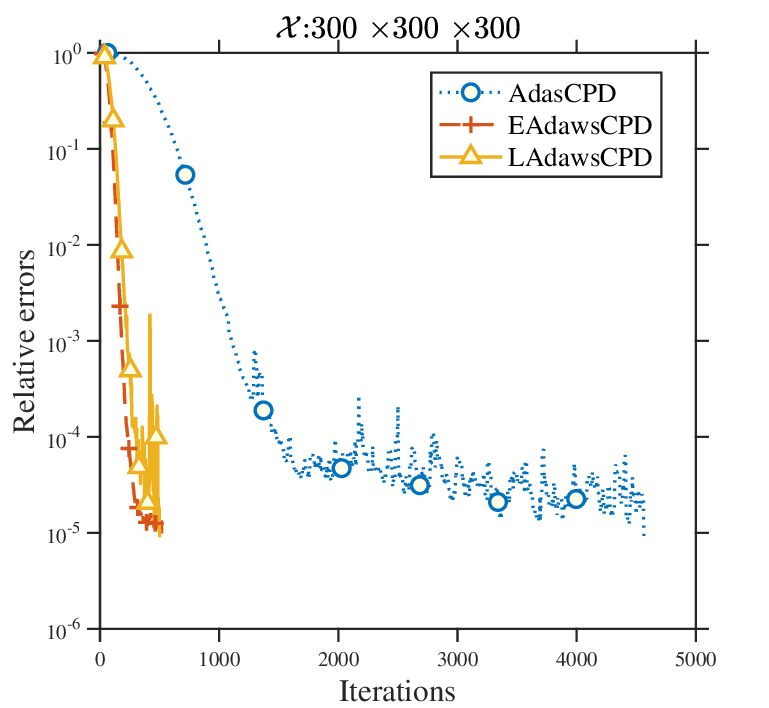}
	\hspace{0.1in}
	\includegraphics[width=0.5\textwidth]{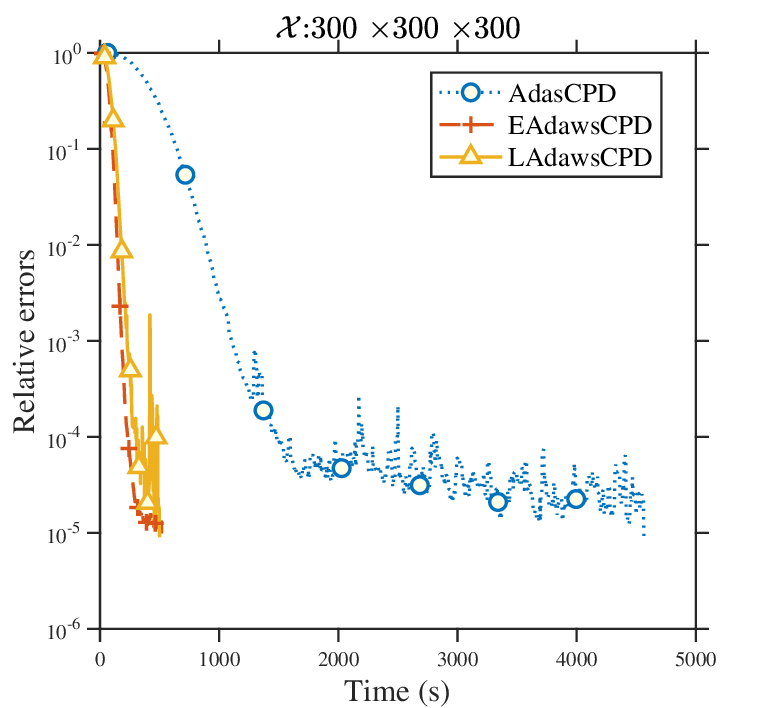}
	\caption{Number of iterations v.s. Relative errors and Time v.s. Relative errors output by the algorithms with $Tol = 10^{-5}$, $|\mc{F}_n|=18$, $R=10$, $noise = 0$, and random initialization for the tensor with $I=300$, $R_{true}=10$, $spread = 45$, and $magnitude = 36$.}
	\label{fig:Ichange300}
\end{figure}

In \Cref{tab:Ichange}, we list the performance of algorithms for tensors with different sizes. From this table, we can see that our algorithms have better performance than \texttt{AdasCPD} in \cite{fu2020BlockRandomizedStochastic} in terms of  the number of iterations and overall running time. Some specific discussions of the performance are in order. 
\begin{enumerate}
	\item With the increase of tensor size, \texttt{EAdawsCPD} and \texttt{LAdawsCPD} always have better performance in terms of the iterations and running time compared with \texttt{AdasCPD}. 
	\item For the probability distribution, the leverage-based algorithm needs a few more iterations than the Euclidean-based algorithm, and the gap is getting larger as the tensor size increases. The reason may be that the latter has a theoretical guarantee that such probability distribution is reasonable \cite{needell2017BatchedStochastic,needell2016StochasticGradient}, while the former has no similar theoretical result and is more of an empirical choice.
	\item Theoretically, the running time of a single step of our algorithms will be a little more than that of \texttt{AdasCPD}. This is because our algorithms use the importance sampling probability distributions which are a little expensive to compute. However, the importance sampling can improve convergence speed and hence can reduce iterations. So, the overall running time of our algorithms is still much less than that of \texttt{AdasCPD}. 
\end{enumerate}

Besides, to make the above comparison more intuitive, we also plot the numerical results for $I=300$ from \Cref{tab:Ichange} in \Cref{fig:Ichange300}, from which we can see that \texttt{AdasCPD} indeed converges slowly and is unstable. Whereas, our algorithms have quite good performance.

\begin{table}[htbp]
	\caption{Performance of the algorithms with $Tol = 10^{-5}$, $|\mc{F}_n| = 18$, different target ranks $R$, $noise = 0$, and random initialization for the tensor generated by $300 \times 10$ factor matrices with $spread = 45$ and $magnitude = 36$.}
	\label{tab:Rchange}
	\begin{tabular}{llllll}
		\hline
		\multicolumn{2}{c}{Algorithms}                                                        & $R=5$     & $R=10$    & $R=15$    & $R=20$    \\ \hline
		\multirow{2}{*}{\texttt{AdasCPD }\cite{fu2020BlockRandomizedStochastic}} & Iterations & 3059.1    & 3105.8    & 4823.7    & 7094.8    \\
		& Seconds    & 480.47423 & 472.43042 & 767.12695 & 1125.8787 \\ \hline
		\multirow{2}{*}{\texttt{EAdawsCPD}}                                      & Iterations & 474       & 390.3     & 406.7     & 824.9     \\
		& Seconds    & 73.963123 & 58.656213 & 65.151505 & 132.01017 \\ \hline
		\multirow{2}{*}{\texttt{LAdawsCPD}}                                      & Iterations & 535.7     & 429.9     & 567.1     & 827.1     \\
		& Seconds    & 83.791824 & 64.193687 & 90.722473 & 132.73588 \\ \hline
	\end{tabular}
\end{table}

In \Cref{tab:Rchange}, we list the performance of algorithms with different target ranks for the same tensor. 
Numerical results show that our algorithms always 
perform much better 
than \texttt{AdasCPD} in iterations and running time. Furthermore, for all the algorithms, the closer the target rank is to the true rank, the better the results are.

To validate the robustness of our algorithms to 
noises, we run additional experiments on the same synthetic tensor with Gaussian noises with different standard deviations. The results, shown in \Cref{tab:Nchange}, indicate that the earlier experiment results are indeed robust to noises. 

\begin{table}[htbp]
	\caption{Performance of the algorithms with $Tol = 10^{-5}$, $|\mc{F}_n| = 18$, the target rank $R = 10$, and random initialization for different tensors generated by $300 \times 10$ factor matrices with $spread = 45$, $magnitude = 36$, and  different noises.}
	\label{tab:Nchange}
	\begin{tabular}{llllll}
		\hline
		\multicolumn{2}{c}{Algorithms}                                                        & $noise=0$  & $noise=0.01$ & $noise=0.1$ & $noise=1$  \\ \hline
		\multirow{2}{*}{\texttt{AdasCPD }\cite{fu2020BlockRandomizedStochastic}} & Iterations & 3105.8     & 3211.7       & 3024.9      & 3511.2     \\
		& Seconds    & 472.430421 & 505.928736   & 481.988384  & 553.33599  \\ \hline
		\multirow{2}{*}{\texttt{EAdawsCPD}}                                      & Iterations & 390.3      & 360.5        & 381.6       & 414.6      \\
		& Seconds    & 58.6562125 & 56.748149    & 60.5743992  & 65.3003044 \\ \hline
		\multirow{2}{*}{\texttt{LAdawsCPD}}                                      & Iterations & 429.9      & 410          & 439         & 452.1      \\
		& Seconds    & 64.1936869 & 64.6965127   & 70.8384762  & 71.6683267 \\ \hline
	\end{tabular}
\end{table}

In addition, as 
in \cite{fu2020BlockRandomizedStochastic}, our algorithms can also be generalized to nonnegative or other constrained situations. 
So, we also present the comparison of the algorithms with nonnegative constraints in \Cref{tab:nonnegative} and \Cref{fig:synthetic2}, from which we can see that the earlier experiment results are also robust to constraints.

\begin{table}[htbp]
	\caption{Performance of the algorithms with $Tol = 10^{-5}$, $|\mc{F}_n| = 18$, the target rank $R = 10$, $noise = 0$,  and random initialization for different tensors generated by $I \times 10$ factor matrices with different $I$ under nonnegative constraint.}
	\label{tab:nonnegative}
	\resizebox{1\linewidth}{!}{
		\begin{tabular}{lllllll}
			\hline
			\multicolumn{2}{c}{\multirow{2}{*}{Algorithms}}                                       & $I=100$                                                                & $I=200$                                                                 & $I=300$                                                                 & $I=400$                                                                 & $I=500$                                                                 \\ \cline{3-7} 
			\multicolumn{2}{c}{}                                                                  & \begin{tabular}[c]{@{}l@{}}$spread=15$, \\ $magnitude=24$\end{tabular} & \begin{tabular}[c]{@{}l@{}}$spread=30$, \\  $magnitude=30$\end{tabular} & \begin{tabular}[c]{@{}l@{}}$spread=45$, \\  $magnitude=36$\end{tabular} & \begin{tabular}[c]{@{}l@{}}$spread=60$, \\  $magnitude=42$\end{tabular} & \begin{tabular}[c]{@{}l@{}}$spread=75$, \\  $magnitude=48$\end{tabular} \\ \hline
			\multirow{2}{*}{\texttt{AdasCPD }\cite{fu2020BlockRandomizedStochastic}} & Iterations & 500.1                                                                  & 910.7                                                                   & 1424.5                                                                  & 2203.7                                                                  & 3372.8                                                                  \\
			& Seconds    & 2.0256353                                                              & 28.981344                                                               & 225.79166                                                               & 737.83326                                                               & 2550.5717                                                               \\ \hline
			\multirow{2}{*}{\texttt{EAdawsCPD}}                                      & Iterations & 156.6                                                                  & 190.6                                                                   & 250.6                                                                   & 319                                                                     & 394.2                                                                   \\
			& Seconds    & 0.6437757                                                              & 6.0335162                                                               & 40.041681                                                               & 110.00088                                                               & 299.55425                                                               \\ \hline
			\multirow{2}{*}{\texttt{LAdawsCPD}}                                      & Iterations & 153.7                                                                  & 205.9                                                                   & 271.5                                                                   & 337.5                                                                   & 437.6                                                                   \\
			& Seconds    & 0.6892614                                                              & 6.6271954                                                               & 43.308729                                                               & 111.84547                                                               & 331.55795                                                               \\ \hline
		\end{tabular}
	}
\end{table}

\begin{figure}[htbp]
	\includegraphics[width=0.5\textwidth]{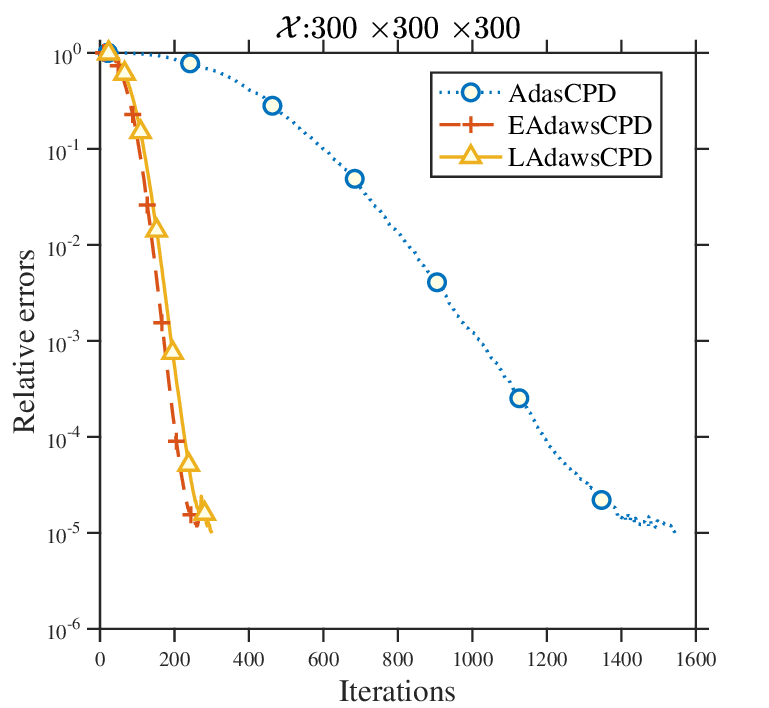}
	\hspace{0.1in}
	\includegraphics[width=0.5\textwidth]{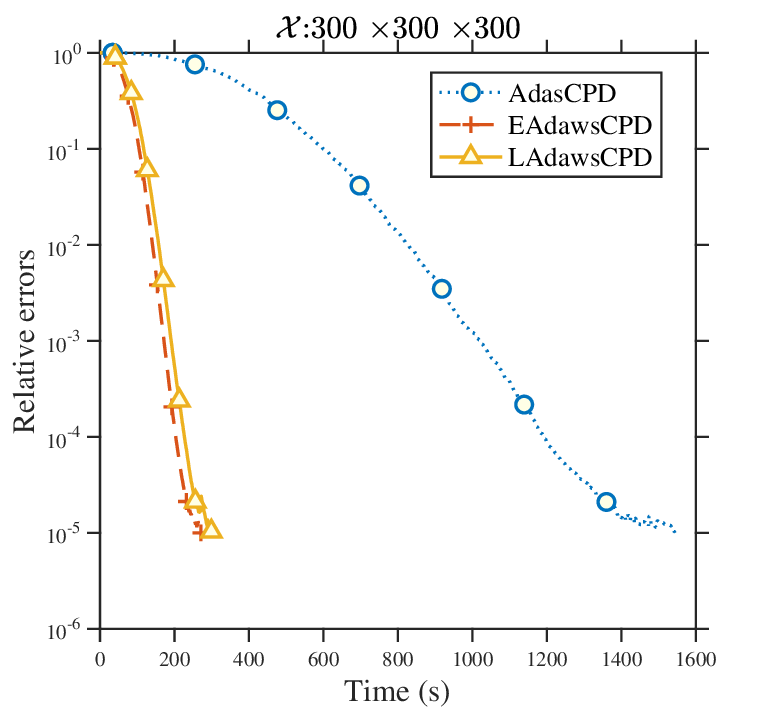}
	\caption{Number of iterations v.s. Relative errors and Time v.s. Relative errors output by the algorithms with 
	$Tol = 10^{-5}$, $|\mc{F}_n|=18$, $R=10$, $noise = 0$, and random initialization for the tensor with $I=300$, $R_{true}=10$, $spread = 45$, $magnitude = 36$, and $\mat{A}^{(n)}\geq\vect{0}$.}
	\label{fig:synthetic2}
\end{figure}

Finally, using the tensor for \Cref{fig:Ichange300} (\textbf{Data I} in short), we compare the variances of stochastic gradients listed in \Cref{tab:var}. Meanwhile, we also consider another two tensors:
\textbf{Data II}, which is generated by three factor matrices of size $300 \times 10$ with independent standard Gaussian entries, and
\textbf{Data III}, which is the same as \textbf{Data II} except that 
one entry of each factor matrices is chosen uniformly and set to 20.

All the numerical results are reported in \Cref{fig:var}, from which we can see that all the variances decrease with the increase 
of iterations, and 
the variances 
related to importance sampling are not worse than the corresponding ones 
based on uniform sampling. Moreover, for 
\textbf{Data I} and \textbf{Data III}, the former  
is much smaller than the latter, and decrease much faster as the number of iterations increases. This is mainly because the factor matrices for forming the tensors have high coherence, i.e., their maximal leverage scores are large. 

\begin{figure}[htbp] 
	\centering 
	\subfloat[Data I]{\includegraphics[scale=0.3]{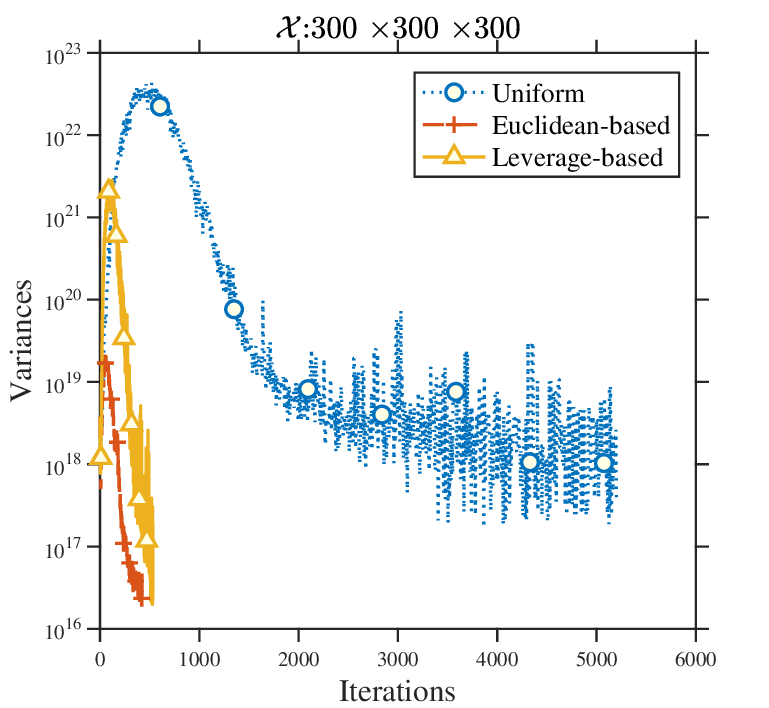}} 
	\subfloat[Data II]{\includegraphics[scale=0.3]{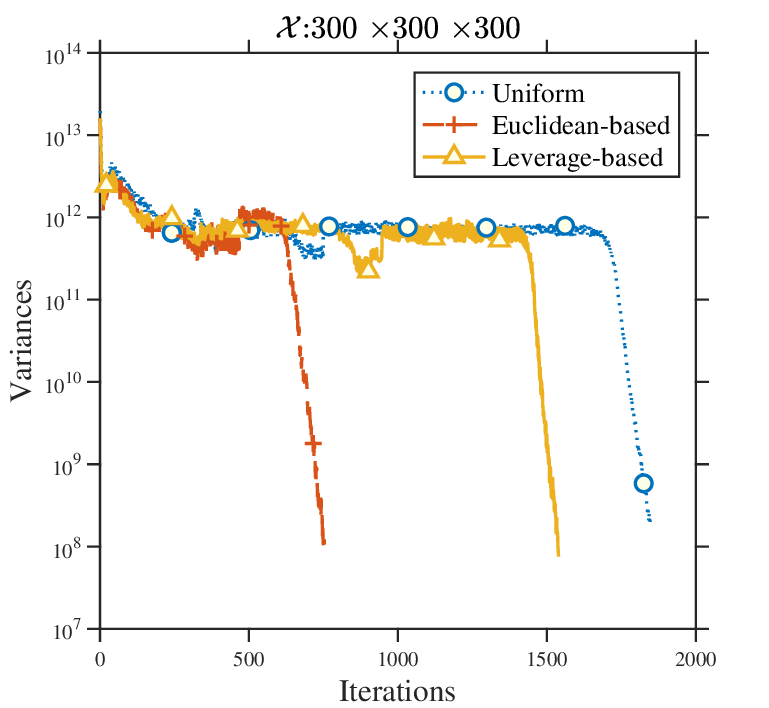}} 
	\subfloat[Data III]{\includegraphics[scale=0.3]{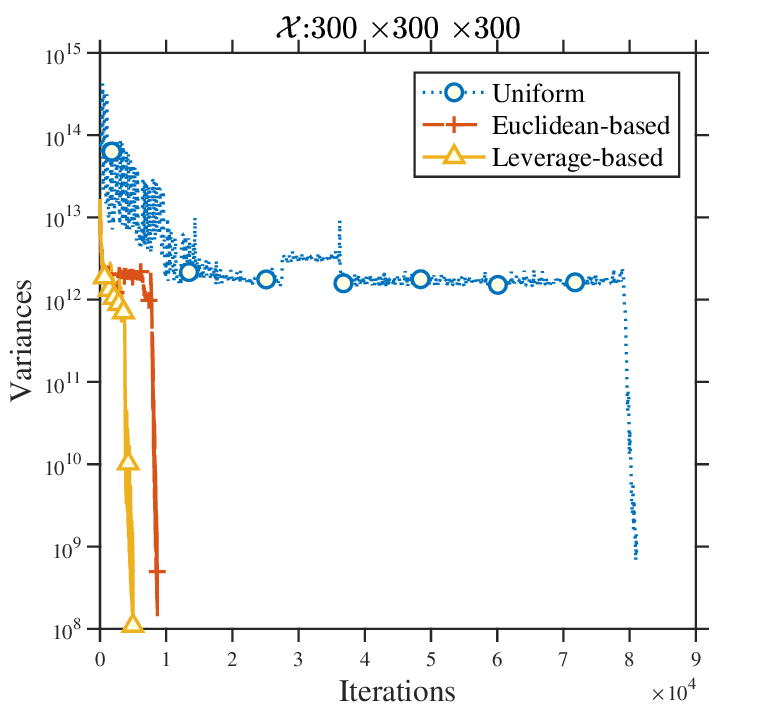}} 
	\caption{Comparison of variances of  stochastic gradients for 
	different tensors.}
	\label{fig:var}
\end{figure}

\subsection{Real data}
\label{sec:realdata}
In this subsection, we test our algorithms on hyperspectral images (HSIs), which are special images with two spatial coordinates and one spectral coordinate. We consider three data tensors  available at \url{http://www.ehu.eus/ccwintco/index.php/Hyperspectral Remote Sensing Scenes}. Their brief information is listed in \Cref{tab:realdata}.

\begin{table}[h] 
	\caption{Size and type of real datasets.} 
	\label{tab:realdata}
	\begin{tabular}{llll} 
		\toprule
		Dataset & Size & Type\\
		\midrule
		SalinasA. & $83 \times 86 \times 224$ & Hyperspectral \\
		Indian Pines & $145 \times 145 \times 220$ & Hyperspectral \\
		Pavia Uni. & $610 \times 340 \times 103$ & Hyperspectral \\
		\bottomrule
	\end{tabular}
\end{table}

\begin{table}[htbp]
	\caption{Performance of algorithms on real datasets.}
	\label{tab:real}
	\begin{tabular}{lllll}
		\hline
		\multicolumn{2}{c}{\multirow{2}{*}{Algorithms}}                                           & SalinasA.                                                                   & Indian Pines                                                                & Pavia Uni.                                                                   \\ \cline{3-5} 
		\multicolumn{2}{c}{}                                                                      & \begin{tabular}[c]{@{}l@{}}$R=10$,\\ $|\mathcal{F}_n|=20$\end{tabular} & \begin{tabular}[c]{@{}l@{}}$R=10$,\\ $|\mathcal{F}_n|=20$\end{tabular} & \begin{tabular}[c]{@{}l@{}}$R=100$, \\ $|\mathcal{F}_n|=20$\end{tabular} \\ \hline
		\multirow{2}{*}{\texttt{AdasCPD} \cite{fu2020BlockRandomizedStochastic}} & $Tol$   & 0.00697493                                                                  & 0.00782241                                                                  & 0.02831117                                                                   \\
		& Seconds & 288.535306                                                                  & 653.276364                                                                  & 4387.12147                                                                   \\ \hline
		\multirow{2}{*}{\texttt{EAdawsCPD}}                                              & $Tol$   & 0.00611473                                                                  & 0.00725646                                                                  & 0.02792415                                                                   \\
		& Seconds & 291.374795                                                                  & 659.02316                                                                   & 4450.21809                                                                   \\ \hline
		\multirow{2}{*}{\texttt{LAdawsCPD}}                                              & $Tol$   & 0.00644397                                                                  & 0.00741898                                                                  & 0.02796972                                                                   \\
		& Seconds & 295.962095                                                                  & 663.648648                                                                  & 4559.80737                                                                   \\ \hline
	\end{tabular}
\end{table}

\begin{figure}[htbp]
	\includegraphics[width=0.5\textwidth]{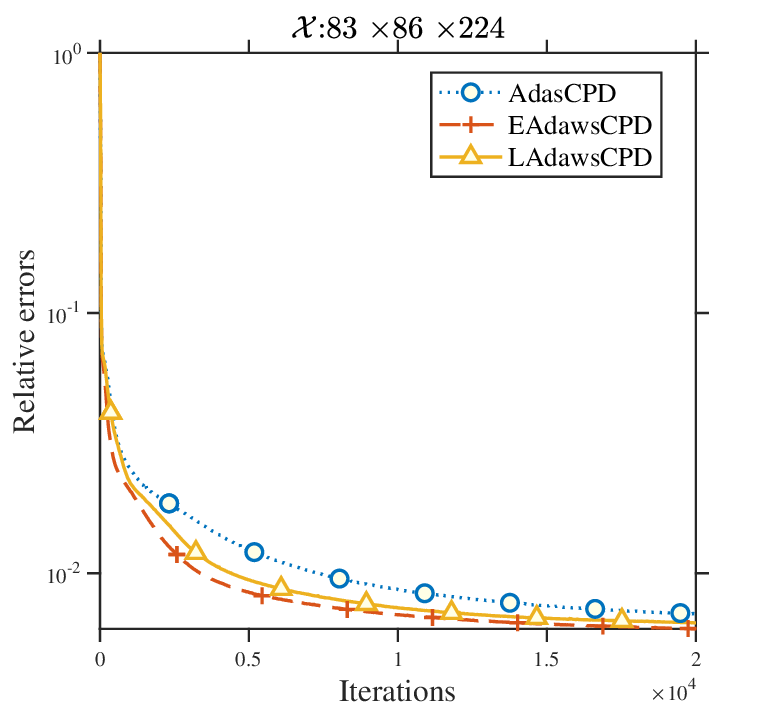}
	\hspace{0.1in}
	\includegraphics[width=0.5\textwidth]{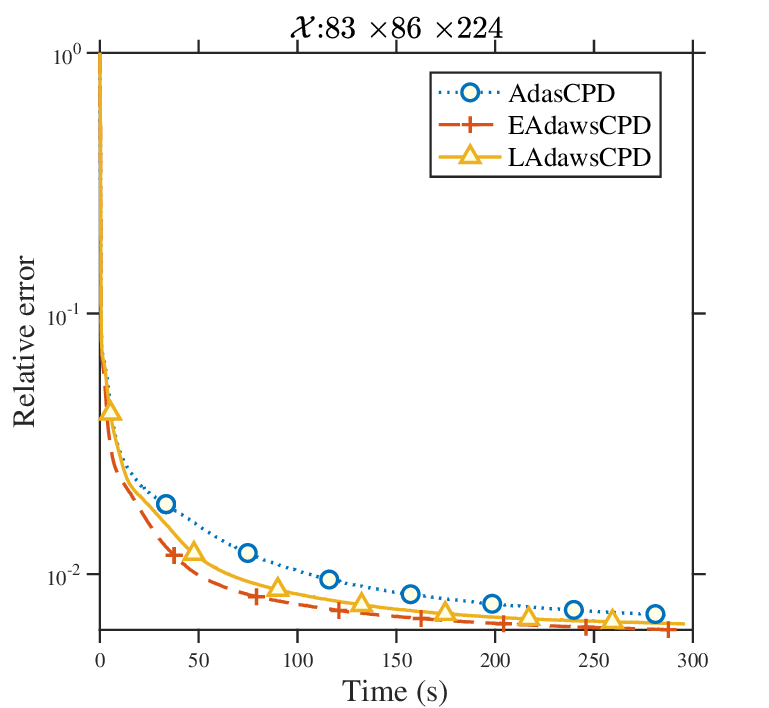}
	\caption{Number of iterations v.s. Relative errors and Time v.s. Relative errors output by the algorithms with $|\mc{F}_n|=20$ and $R=10$ for 
	SalinasA..}
	\label{fig:real}
\end{figure}

\Cref{tab:real} shows the relative errors and running time of algorithms for these  
tensors under different target ranks. The numerical results are returned after 20000 iterations (30000 iterations for Pavia Uni.) with the standard Gaussian matrices being the initial factor matrices. We also plot the results of SalinasA. from \Cref{tab:real} in \Cref{fig:real}.
It is seen that our algorithms still outperform \texttt{AdasCPD}. However, in contrast to the results for synthetic data, the differences between them are not very remarkable. 
The main reason may be that these real data are very 
even. 
To illustrate, we plot the leverage scores and coherence of the coefficient matrices $\mat{Z}^{(n)}$ of CP-ALS subproblems in \Cref{fig:reason_lev}. For comparison, we also plot the corresponding results for \textbf{Data I} mentioned above. 
Examining the vertical coordinates of \Cref{fig:reason_lev-SalinasA,fig:reason_lev-IndianP,fig:reason_lev-PaviaU,fig:reason_lev-Kolda} attentively, we can observe that the leverage scores for the three real datasets oscillate in a narrower range compared with \textbf{Data I}. Additionally, \Cref{fig:reason_lev-bar} shows that the coherence of these real data are indeed not very high. 
In addition, 
we also plot the squared Euclidean norms of the rows of the coefficient matrices and their maximum values in \Cref{fig:reason_euc}\footnote{In \Cref{fig:reason_euc-bar}, 
we take the logarithm of the vertical coordinate to make the figure much clearer and easier to compare. 
}. The findings are similar. 

\begin{figure}[htbp] 
	\centering 
	\subfloat[SalinasA.]{\includegraphics[scale=0.22]{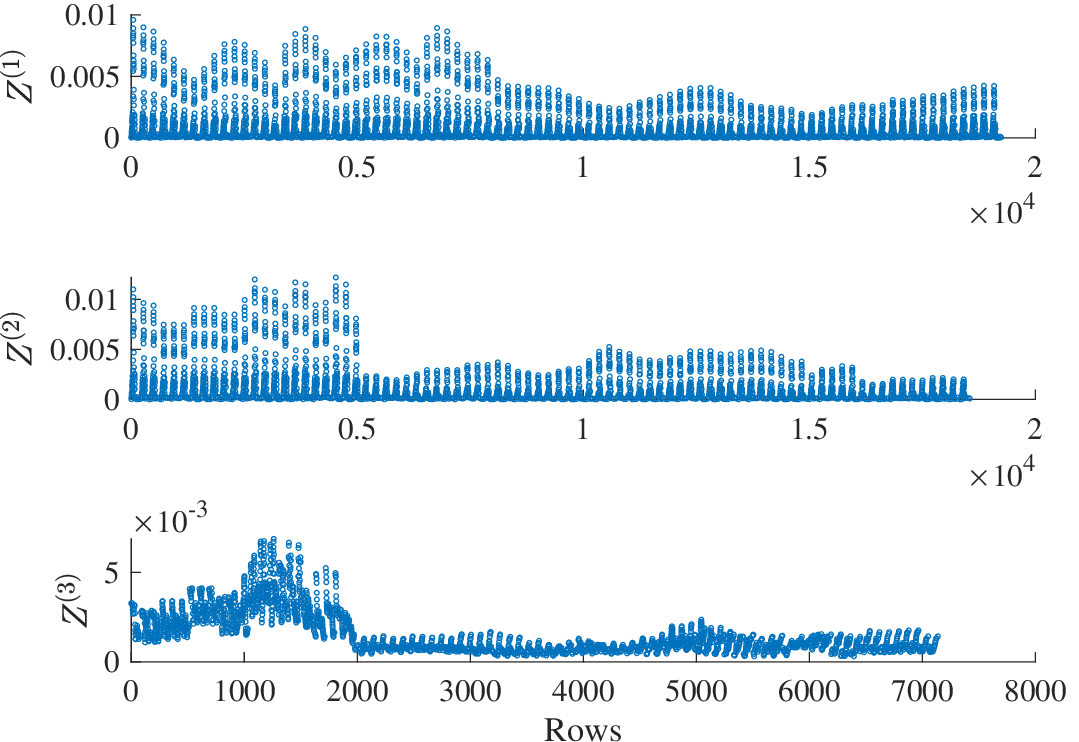} \label{fig:reason_lev-SalinasA}} 
	\subfloat[Indian Pines]{\includegraphics[scale=0.22]{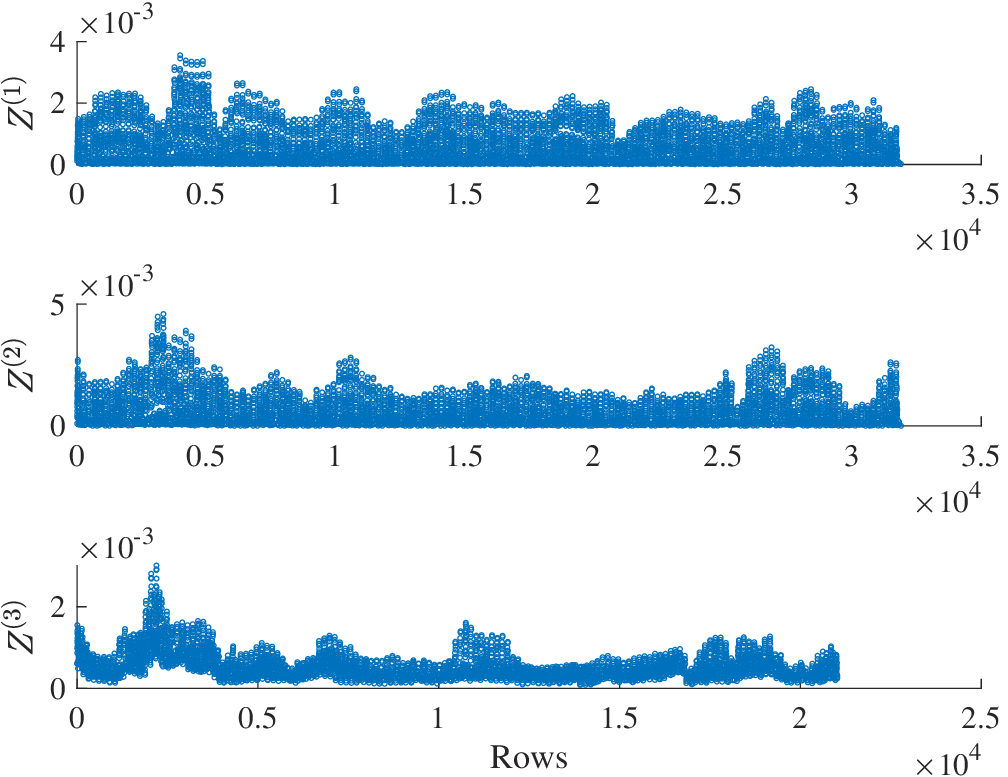} \label{fig:reason_lev-IndianP}} 
	\subfloat[Pivia Uni.]{\includegraphics[scale=0.22]{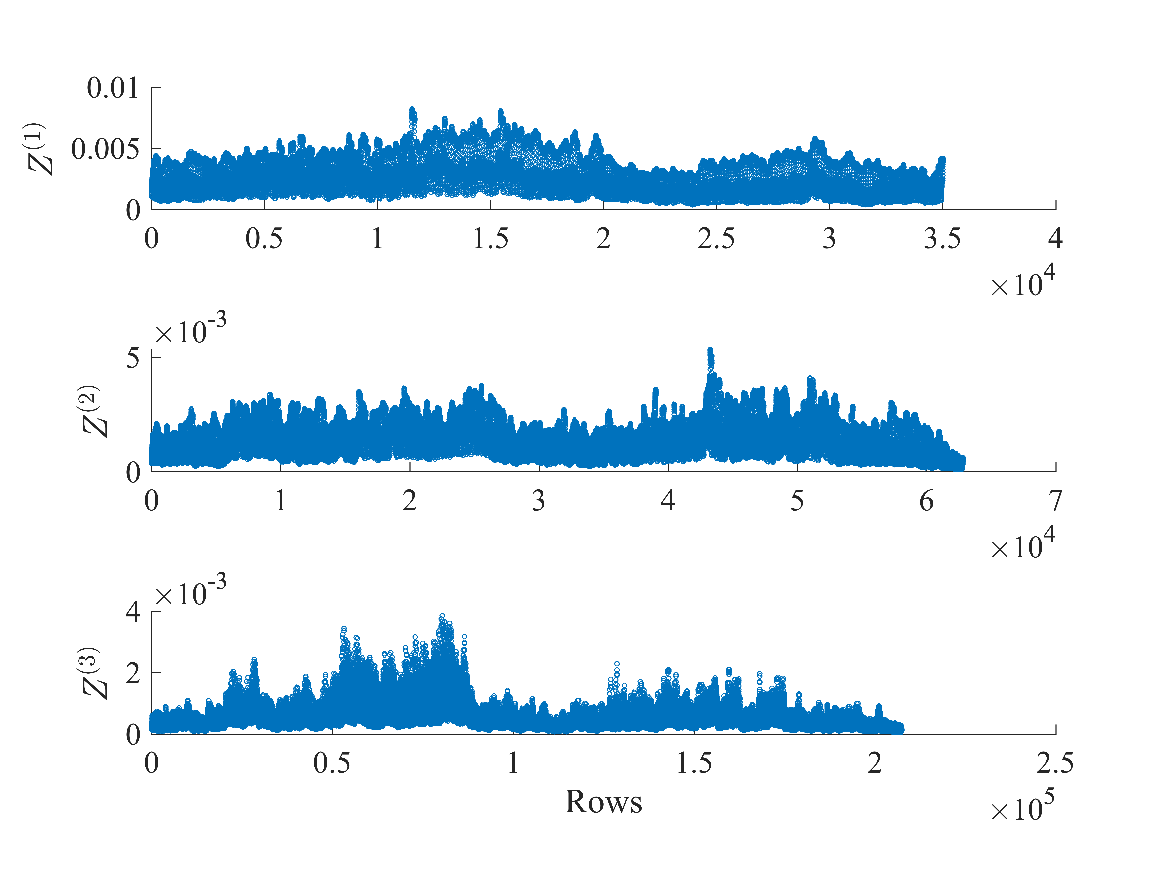}\label{fig:reason_lev-PaviaU}} 
	\quad 
	\subfloat[Data I]{\includegraphics[scale=0.3]{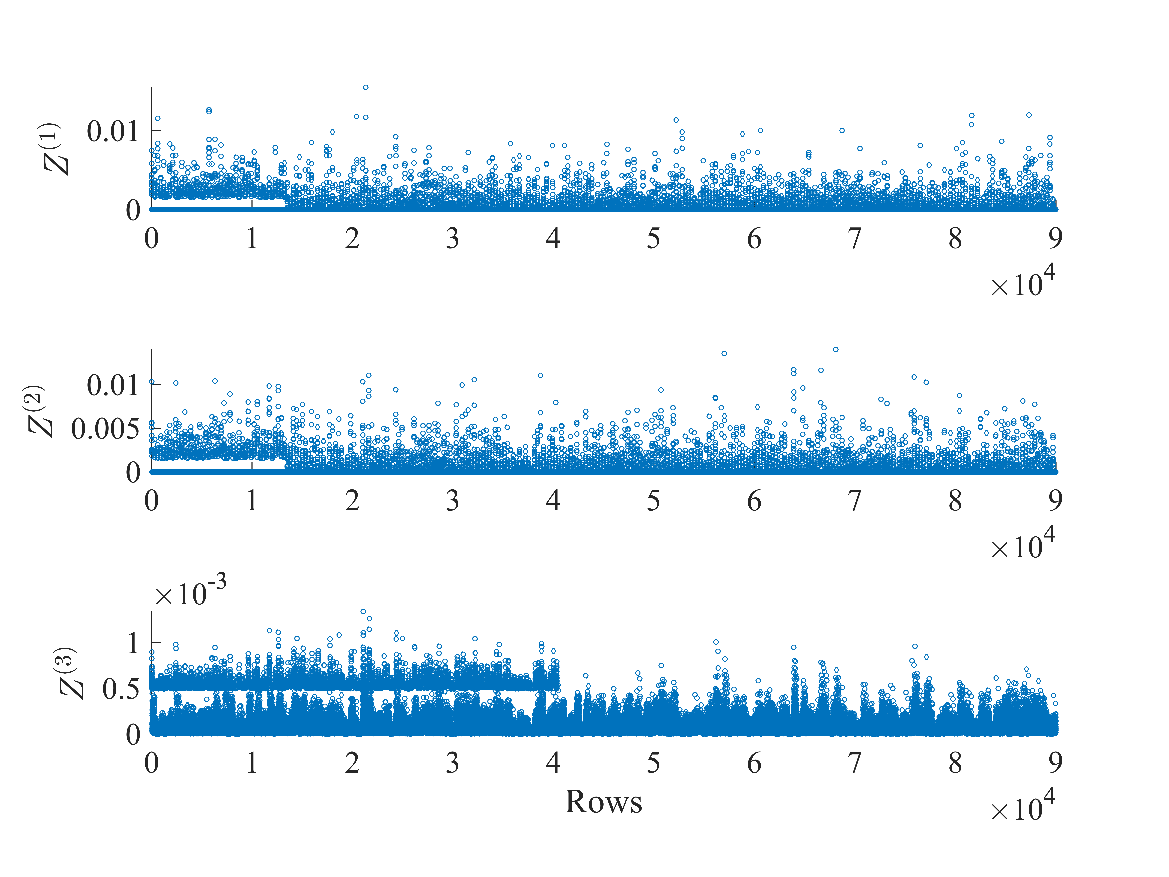}\label{fig:reason_lev-Kolda}} 
	\subfloat[Bar chart]{\includegraphics[scale=0.3]{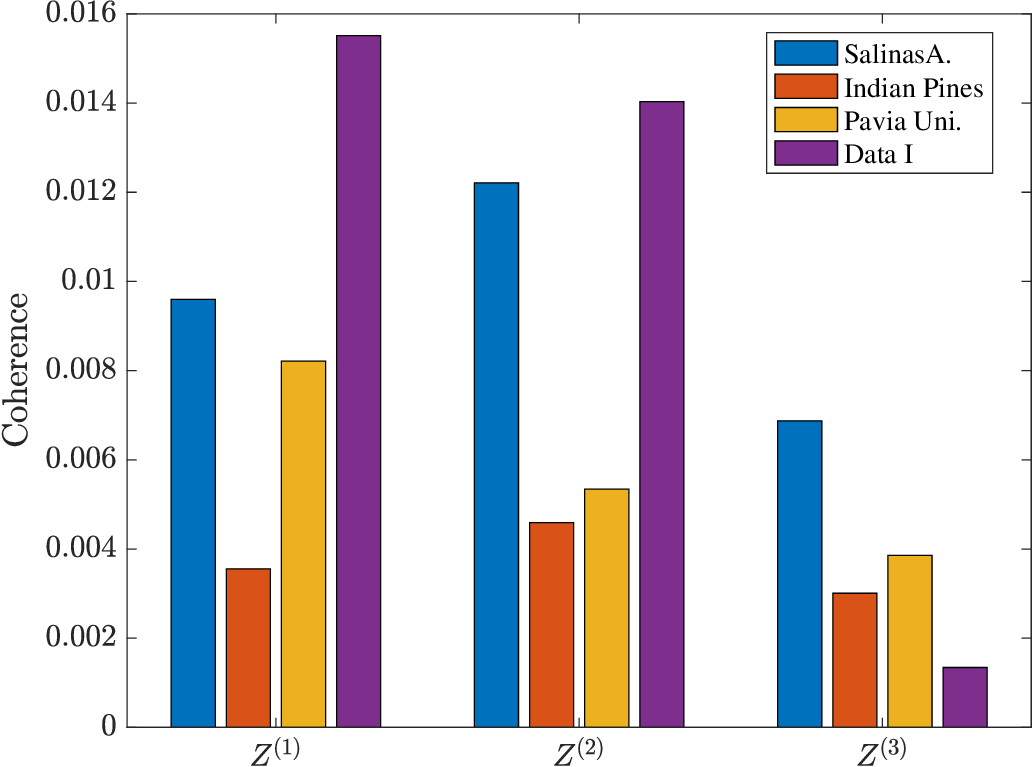}\label{fig:reason_lev-bar}} 
	\caption{(a)-(d) Leverage scores for 
	four different tensors. (e) Coherence for 
	four different tensors.}
	\label{fig:reason_lev}
\end{figure} 

\begin{figure}[htbp] 
	\centering 
	\subfloat[SalinasA.]{\includegraphics[scale=0.22]{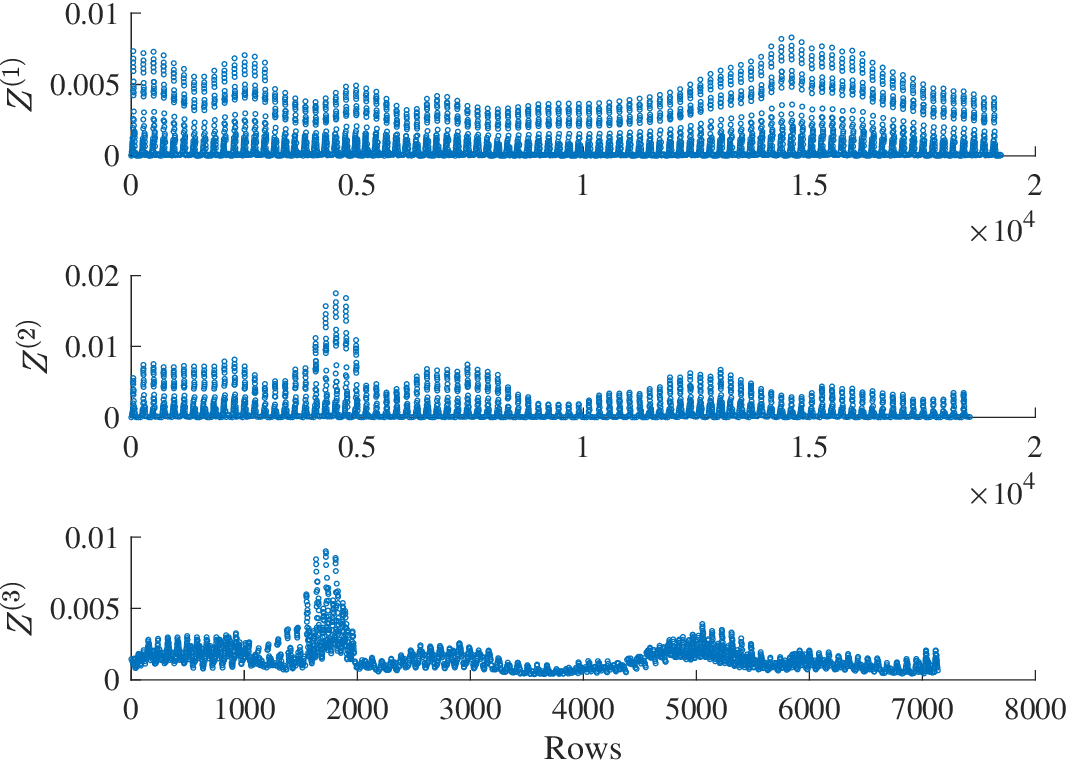}} 
	\subfloat[Indian Pines]{\includegraphics[scale=0.22]{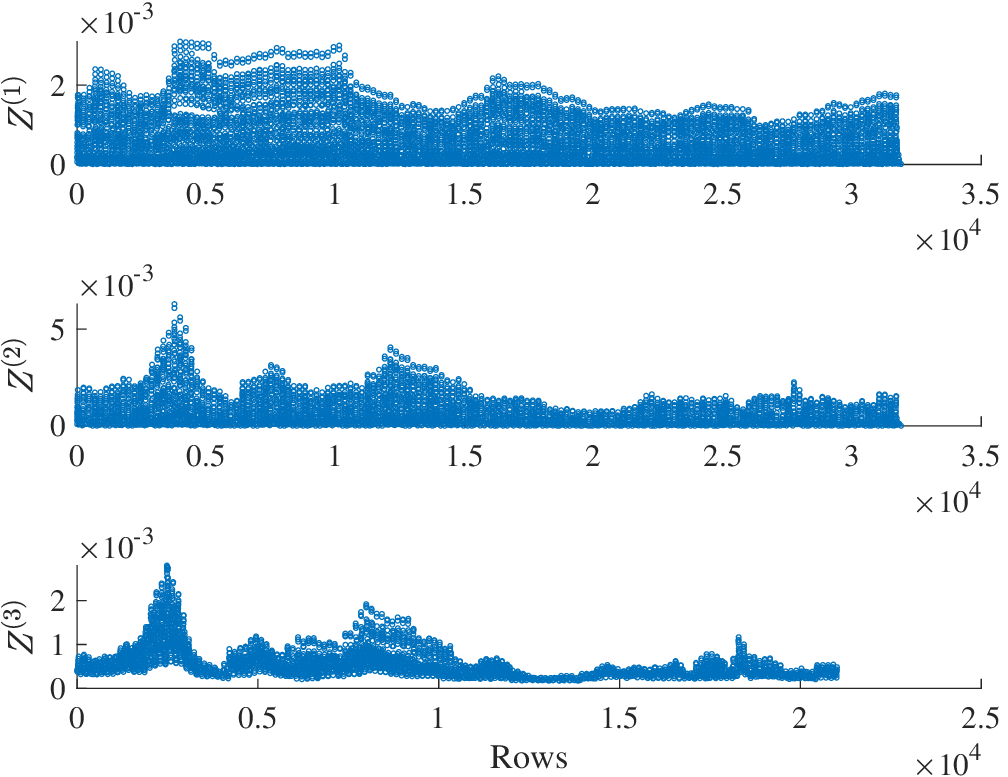}} 
	\subfloat[Pivia Uni.]{\includegraphics[scale=0.22]{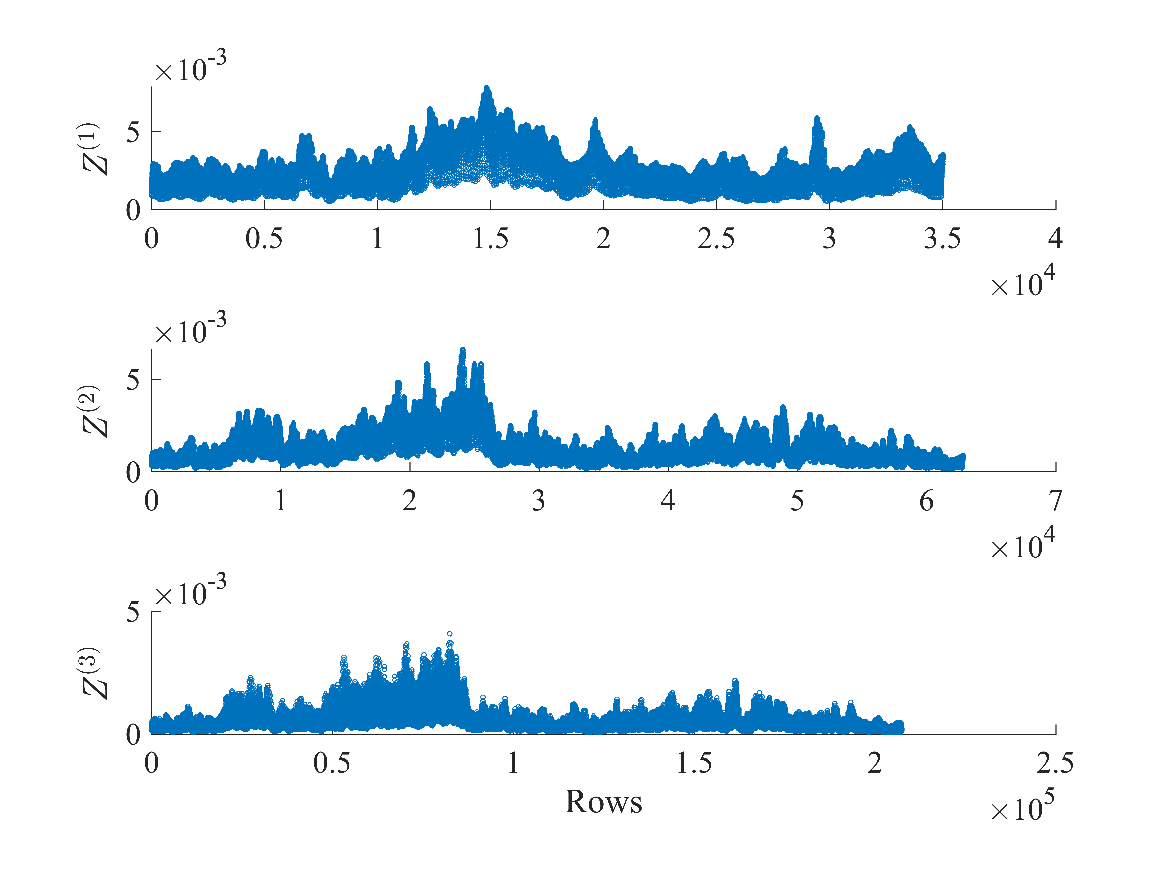}} 
	\quad 
	\subfloat[Data I]{\includegraphics[scale=0.3]{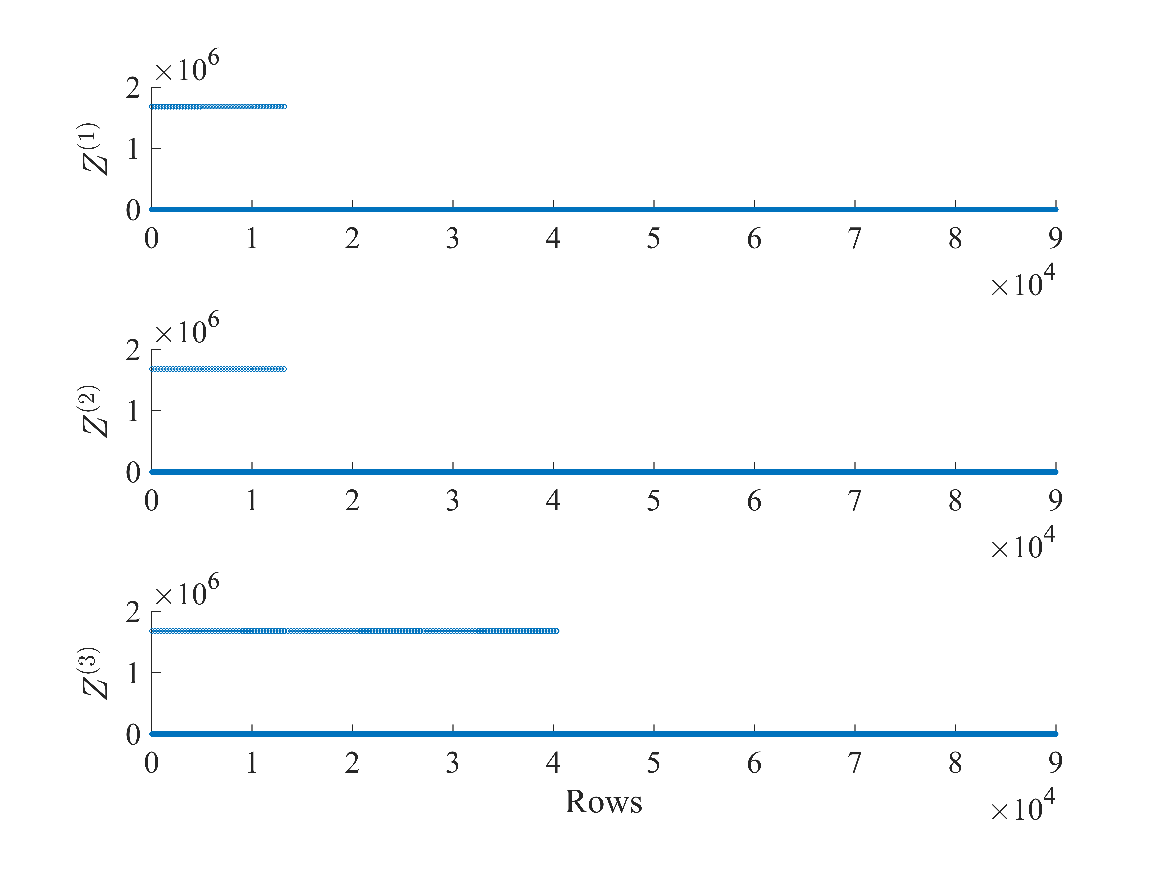}} 
	\subfloat[Bar chart]{\includegraphics[scale=0.3]{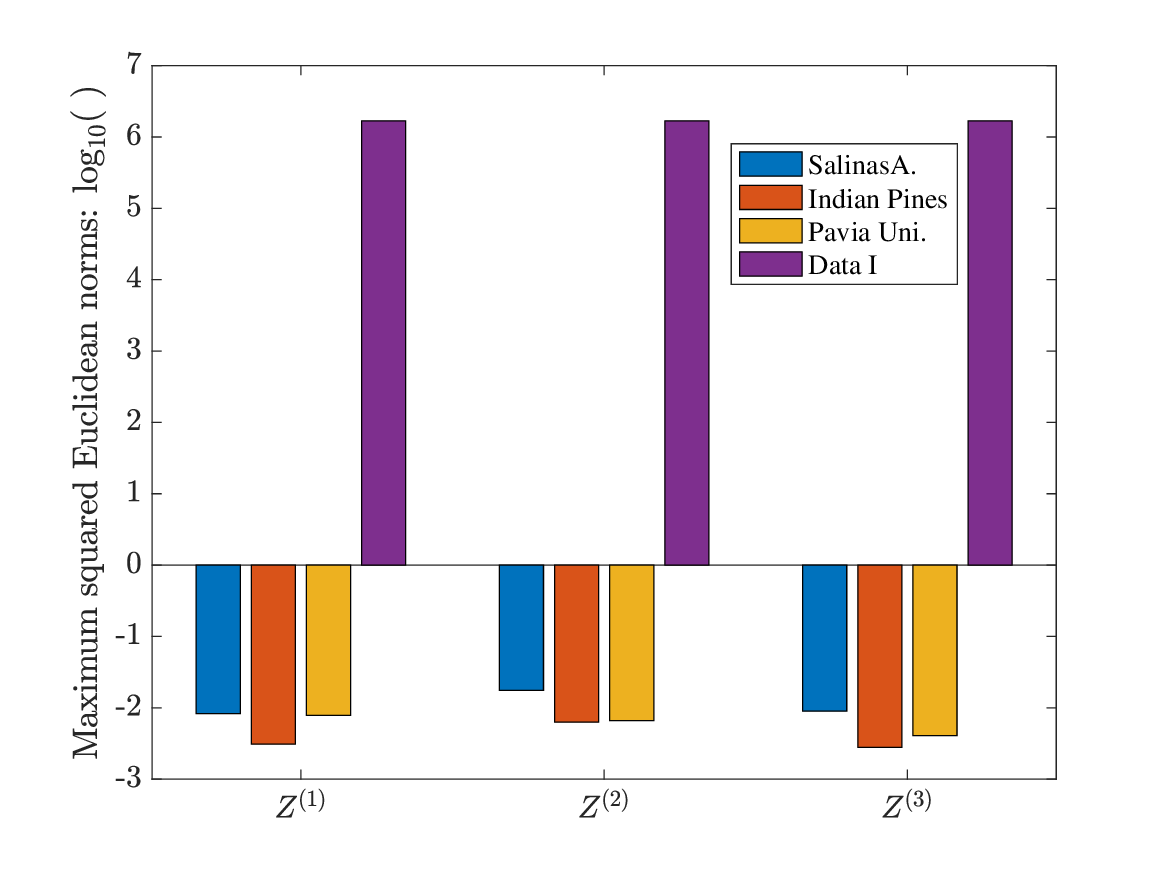} \label{fig:reason_euc-bar} } 
	\caption{(a)-(d) Squared Euclidean norms for 
	four different tensors. (e) Maximum squared Euclidean norms for 
	four different tensors.}
	\label{fig:reason_euc}
\end{figure}

\section{Concluding Remarks}
\label{sec:Conclusion}
In this paper, based on two different sampling strategies, we propose a block-randomized gradient descent method with importance sampling for CP decomposition, and provide its detailed theoretical analysis. 
Numerical experiments show that our method always outperforms the existing one, \texttt{AdasCPD}. As done in \cite{wang2021momentum}, it is interesting to consider the momentum version of our method. Also, it is interesting to introduce the greedy sampling strategies \cite{bai2018GreedyRandomized,zhang2022GreedyMotzkin} and the scaling technique \cite{tong2021AcceleratingIllconditioned} into our method. We leave them for future research.


\subsection*{Funding}
This work was supported by the National Natural Science Foundation of China (No. 11671060) and the Natural Science Foundation Project of CQ CSTC (No. cstc2019jcyj-msxmX0267).

\subsection*{Data Availability}
The data that support the findings of this study are available from the corresponding author upon reasonable request.

\subsection*{Competing Interests}
The authors declare that they have no conflict of interest.

\bibliography{sn-bibliography-yzq}

\end{sloppypar}
\end{document}